\documentclass[onecolumn]{article}

\usepackage{graphicx,epsfig}
\usepackage{times}
\usepackage{hyperref}
\usepackage{amssymb, amsthm}
\usepackage {amssymb,latexsym,epic,eepic,stmaryrd,url}
\newtheorem{thm}{Theorem}
\newtheorem{clm}{Claim}
\newtheorem{fact}{Fact}
\newtheorem{lem}{Lemma}
\newtheorem{asm}{Assumption}
\newtheorem{obs}{Observation}
\newcommand{\bqed}{\mbox{}\hfill$\blacksquare$\vspace{2mm}\\}

\begin{document}
\title{\Large \bf Acyclic Edge Coloring of 2-degenerate Graphs}

\author{Manu Basavaraju\thanks{Computer Science and Automation department,
Indian Institute of Science,
Bangalore- 560012,
India.  {\tt manu@csa.iisc.ernet.in}} \and L. Sunil Chandran\thanks{Computer Science and Automation department,
Indian Institute of Science,
Bangalore- 560012,
India.  {\tt sunil@csa.iisc.ernet.in}}
}

\date{}
\pagestyle{plain}
\maketitle

\begin{abstract}

An $acyclic$ edge coloring of a graph is a proper edge coloring such that there are no bichromatic cycles. The \emph{acyclic chromatic index} of a graph is the minimum number k such that there is an acyclic edge coloring using k colors and is denoted by $a'(G)$.  A graph is called $2$-$degenerate$ if any of its induced subgraph has a vertex of degree at most 2. The class of $2$-$degenerate\ graphs$ properly contain $series$-$parallel\ graphs$, $outerplanar\ graphs$, \emph{non-regular subcubic graphs}, \emph{planar graphs of girth at least 6} and \emph{circle graphs of girth at least 5} as subclasses. It was conjectured by Alon, Sudakov and Zaks (and much earlier by Fiamcik) that $a'(G)\le \Delta+2$, where $\Delta =\Delta(G)$ denotes the maximum degree of the graph. We prove the conjecture for $2$-$degenerate$ graphs. In fact we prove a stronger bound: we prove that if $G$ is a 2-degenerate graph with maximum degree $\Delta$, then $a'(G)\le \Delta + 1$.

\end{abstract}


\section{Introduction}

All graphs considered in this paper are finite and simple. A proper \emph{edge coloring} of $G=(V,E)$ is a map $c: E\rightarrow C$ (where $C$ is the set of available $colors$ ) with $c(e) \neq c(f)$ for any adjacent edges $e$,$f$. The minimum number of colors needed to properly color the edges of $G$, is called the chromatic index of $G$ and is denoted by $\chi'(G)$. A proper edge coloring c is called acyclic if there are no bichromatic cycles in the graph. In other words an edge coloring is acyclic if the union of any two color classes induces a set of paths (i.e., linear forest) in $G$. The \emph{acyclic edge chromatic number} (also called \emph{acyclic chromatic index}), denoted by $a'(G)$, is the minimum number of colors required to acyclically edge color $G$. The concept of \emph{acyclic coloring} of a graph was introduced by Gr\"unbaum \cite{Grun}. The \emph{acyclic chromatic index} and its vertex analogue can be used to bound other parameters like \emph{oriented chromatic number} and \emph{star chromatic number} of a graph, both of which have many practical applications, for example, in wavelength routing in optical networks ( \cite{ART}, \cite{KSZ} ). Let $\Delta=\Delta(G)$ denote the maximum degree of a vertex in graph $G$. By Vizing's theorem, we have $\Delta \le \chi'(G) \le \Delta +1 $(see \cite{Diest} for proof). Since any acyclic edge coloring is also proper, we have $a'(G)\ge\chi'(G)\ge\Delta$. \newline

It has been conjectured by Alon, Sudakov and Zaks \cite{ASZ} (and much earlier by Fiamcik \cite{Fiam}) that $a'(G)\le\Delta+2$ for any $G$. Using probabilistic arguments Alon, McDiarmid and Reed \cite{AMR} proved that $a'(G)\le60\Delta$. The best known result up to now for arbitrary graph, is by Molloy and Reed  \cite{MolReed} who showed that $a'(G)\le16\Delta$. Muthu, Narayanan and Subramanian \cite{MNS1} proved that $a'(G)\le4.52\Delta$ for graphs $G$ of girth at least 220 (\emph{Girth} is the length of a shortest cycle in a graph).\newline

Though the best known upper bound for general case is far from the conjectured $\Delta+2$, the conjecture has been shown to be true for some special classes of graphs. Alon, Sudakov and Zaks \cite{ASZ} proved that there exists a constant $k$ such that $a'(G)\le\Delta+2$ for any graph $G$ whose girth is at least $k\Delta\log\Delta$. They also proved that $a'(G)\le\Delta+2$ for almost all $\Delta$-regular graphs. This result was improved by Ne\v set\v ril and Wormald \cite{NesWorm} who showed that for a random $\Delta$-regular graph $a'(G)\le \Delta+1$. Muthu, Narayanan and Subramanian proved the conjecture for grid-like graphs \cite{MNS2}. In fact they gave a better bound of $\Delta+1$ for these class of graphs. From Burnstein's \cite{Burn} result it follows that the conjecture is true for subcubic graphs. Skulrattankulchai \cite{Skul} gave a polynomial time algorithm to color a subcubic graph using $\Delta+2 = 5$ colors.

Determining $a'(G)$ is a hard problem both from theoretical and algorithmic points of view. Even for the simple and highly structured class of complete graphs, the value of $a'(G)$ is still not determined exactly. It has also been shown by Alon and Zaks \cite{AZ} that determining whether $a'(G)\le3$ is NP-complete for an arbitrary graph $G$. The vertex version of this problem has also been extensively studied ( see \cite{Grun}, \cite{Burn}, \cite{Boro}). A generalization of the acyclic edge chromatic number has been studied: The \emph{r-acyclic edge chromatic number} $a'_r(G)$ is the minimum number of colors required to color the edges of the graph $G$ such that every cycle $C$ of $G$ has at least min\{$\vert C \vert$,$r$\} colors ( see \cite{GeRa}, \cite{GrePi}).

~~~~~~~~~

\noindent\textbf{Our Result:}We prove the conjecture for 2-degenerate graphs. A graph $G$ is called $k$-$degenerate$ if any induced subgraph of $G$, has a vertex of degree at most $k$. For example, planar graphs are 5-degenerate, forests are 1-degenerate. The earliest result on acyclic edge coloring of 2-degenerate graphs was by Card and Roditty \cite{CRod}, where they proved that $a'(G)\le\Delta + k-1$, where $k$ is the maximum edge connectivity, defined as \begin{math} k = \max_{u,v \in V(G)} \lambda(u,v)\end{math} , where $\lambda(u,v)$ is the edge- connectivity of the pair u,v. Note that here $k$ can be as high as $\Delta$. Muthu,Narayanan and Subramanian \cite{MNS3} proved that $a'(G)\le \Delta+1$ for outerplanar graphs which are a subclass of 2-degenerate graphs and posed the problem of proving the conjecture for 2-degenerate graphs as an open problem. In fact they have informed us that very recently they have also derived an upper bound of $\Delta +1$ for series-parallel graphs \cite{MNS4}, which is a slightly bigger subclass of 2-degenerate graphs. Connected non-regular subcubic graphs are 2-degenerate graphs with $\Delta = 3$. Recently Basavaraju and Chandran \cite{MBSC} proved that connected non-regular subcubic graph can be acyclically edge colored using $\Delta +1 = 4$ colors. Another two interesting subclasses of 2-degenerate graphs are \emph{planar graphs of girth 6} and \emph{circle graphs of girth 5} \cite{Ageev}. As far as we know, nothing much is known about the acyclic edge chromatic number of these graphs. In this paper, we prove the following theorem,

\begin{thm}
\label{thm:thm1}
Let $G$ be a 2-degenerate graph with maximum degree $\Delta$, then $a'(G)\le \Delta + 1$.
\end{thm}

Our result is tight since there are 2-degenrate graphs which require $\Delta+1$ colors (e.g., cycle, non-regular subcubic graphs, etc.). Most of the work in this field has been nonconstructive, using probabilistic methods. In contrast, our proof is constructive and yields an efficient polynomial time algorithm. It is easy to see that its complexity is $O(\Delta n^2)$. (We have presented the proof in a non-algorithmic way. But it is easy to extract the underlying algorithm from it.)

\noindent \textbf{Remark: }It may be noted that though $a'(G) \le \Delta+1$ for 2-degenerate graphs, it is not so in general. In fact every $\Delta$-regular graph on $2n$ vertices with $\Delta > n$ requires at least $\Delta+2$ colors to be acyclically edge colored (See \cite{MBSCMK}).

\section{Preliminaries}

Let $G=(V,E)$ be a simple, finite and connected 2-degenerate graph of $n$ vertices and $m$ edges. Let $x \in V$. Then $N_{G}(x)$ will denote the neighbours of $x$ in $G$. For an edge $e \in E$, $G-e$ will denote the graph obtained by deletion of the edge $e$. For $x,y \in V$, when $e=(x,y)=xy$, we may use $G-\{xy\}$ instead of $G-e$. Let $c:E\rightarrow \{1,2,\ldots,k\}$ be an \emph{acyclic edge coloring} of $G$. For an edge $e\in E$, $c(e)$ will denote the color given to $e$ with respect to the coloring $c$. For $x,y \in V$, when $e=(x,y)=xy$ we may use $c(x,y)$ instead of $c(e)$. For $S \subseteq V$, we denote the induced subgraph on $S$ by $G[S]$.

~~~~~~

\noindent \textbf{Partial Coloring:} Let H be a subgraph of $G$. Then an acyclic edge coloring $c'$ of $H$ is also a partial coloring of $G$. Note that $H$ can be $G$ itself. Thus a coloring $c$ of $G$ itself can be considered a partial coloring. A coloring $c$ of $G$ is said to be a proper partial coloring if $c$ is proper. A proper partial coloring $c$ is called acyclic if there are no bichromatic cycles in the graph. Sometimes we also use the word valid coloring instead of acyclic coloring. Note that with respect to a partial coloring $c$, $c(e)$ may not be defined for an edge $e$. So, whenever we use $c(e)$, we are considering an edge $e$ for which $c(e)$ is defined, though we may not always explicitly mention it.

Let $c$ be a partial coloring of $G$. We denote the set of colors in the partial coloring $c$ by $C = \{1,2,\ldots,\Delta+1\}$. For any vertex $u \in V(G)$, we define $F_u(c) =\{c(u,z) \vert z \in N_{G}(u)\}$. For an edge $ab \in E$, we define $S_{ab}(c) = F_b - \{c(a,b)\}$. Note that $S_{ab}(c)$ need not be the same as $S_{ba}(c)$. We will abbreviate the notation to $F_u$ and $S_{ab}$ when the coloring $c$ is understood from the context.

To prove the main result, we plan to use contradiciton. Let $G$ be the minimum counter example for the statement in $Theorem$ \ref{thm:thm1}.  Let $G =(V,E)$ be a graph on $m$ edges where $m \ge 1$. We will remove an edge $e$ from $G$ and get a graph $G'=(V,E')$. By the minimality of $G$, the graph $G'$ will have an acyclic edge coloring $c:E'\rightarrow \{1,2,\ldots,\Delta +1\}$. Our intention will be to extend the coloring $c$ of $G'$ to $G$ by assigning an appropriate color for the edge $e$ thereby contrdicting the assumption that $G$ is a minimum couter example.

~~~~~

The following defintions arise out of our attempt to understand what may prevent us from extending a partial coloring of $G-e$ to $G$.

\noindent \textbf{Maximal bichromatic Path:} An ($\alpha$,$\beta$)-maximal bichromatic path with respect to a partial coloring $c$ of $G$ is a maximal path consisting of edges that are colored using the colors $\alpha$ and $\beta$ alternatingly. An ($\alpha$,$\beta$,$a$,$b$)-maximal bichromatic path is an ($\alpha$,$\beta$)-maximal bichromatic path which starts at the vertex $a$ with an edge colored $\alpha$ and ends at $b$. We emphasize that the edge of the ($\alpha$,$\beta$,$a$,$b$)-maximal bichromatic path incident on vertex $a$ is colored $\alpha$ and the edge incident on vertex $b$ can be colored either $\alpha$ or $\beta$. Thus the notations ($\alpha$,$\beta$,$a$,$b$) and ($\alpha$,$\beta$,$b$,$a$) have different meanings. Also note that any maximal bichromatic path will have at least two edges. The following fact is obvious from the definition of proper edge coloring:

\begin{fact}
\label{fact:fact1}
Given a pair of colors $\alpha$ and $\beta$ of a proper coloring $c$ of $G$, there can be at most one maximal ($\alpha$,$\beta$)-bichromatic path containing a particular vertex $v$, with respect to $c$.
\end{fact}

A color $\alpha \neq c(e)$ is a \emph{candidate} for an edge \emph{e} in $G$ with respect to a partial coloring $c$ of $G$ if none of the adjacent edges of \emph{e} are colored $\alpha$. A candidate color $\alpha$ is \emph{valid} for an edge \emph{e} if assigning the color $\alpha$ to \emph{e} does not result in any bichromatic cycle in $G$.

Let $e=(a,b)$ be an edge in $G$. Note that any color $\beta \notin F_a \cup F_b$ is a candidate color for the edge $ab$ in $G$ with respect to the partial coloring $c$ of $G$. But $\beta$ may not be valid. What may be the reason? It is clear that color $\beta$ is not $valid$ if and only if there exists $\alpha \neq \beta$ such that a ($\alpha$,$\beta$)-bichromatic cycle gets formed if we assign color $\beta$ to the edge $e$. In other words, if and only if, with respect to coloring $c$ of $G$ there existed a ($\alpha$,$\beta$,$a$,$b$) maximal bichromatic path with $\alpha$ being the color given to the first and last edge of this path. Such paths play an important role in our proof. We call them $critical\ paths$. It is formally defined below:

~~~~~~~~~

\noindent\textbf{Critical Path:} Let $ab \in E$ and $c$ be a partial coloring of $G$. Then a $(\alpha,\beta,$a$,$b$)$ maximal bichromatic path which starts out from the vertex $a$ via an edge colored  $\alpha$ and ends at the vertex $b$ via an edge colored $\alpha$ is called an $(\alpha,\beta,ab)$ critical path. Note that any critical path will be of odd length. Moreover the smallest length possible is three.

~~~~~~~~~

An obvious strategy to extend a valid partial coloring $c$ of $G$ would be to try to assign one of the candidate colors to an uncolored edge $e$. The condition that a candidate color being not valid for the edge $e$ is captured in the following fact.

\begin{fact}
\label{fact:fact2}
Let $c$ be a partial coloring of $G$. A candidate color $\beta$ is not $valid$ for the edge $e=(a,b)$ if and only if $\exists \alpha \in S_{ab} \cap S_{ba}$ such that there is a $(\alpha,\beta,ab)$  critical path in $G$ with respect to the coloring $c$.
\end{fact}

~~~~~~~~~

If all the candidate colors turn out to be $invalid$, we try to $slightly\ modify$ the partial coloring $c$ in such a way that with respect to the modified coloring, one of the candidate colors becomes valid. An obvious way to modify is to recolor an edge so that some critical paths are $broken$ and a candidate color becomes valid. Sometimes we resort to a slightly more sophisticated strategy to modify the coloring namely $color\ exchange$ defined below:

~~~~~~~

\noindent \textbf{Color Exchange:} Let $c$ be a partial coloring of $G$. Let $u,i,j \in V(G)$ and $ui,uj \in E(G)$. We define $Color\ Exchange$ with respect to the edge $ui$ and $uj$, as the modification of the current partial coloring $c$ by exchanging the colors of the edges $ui$ and $uj$ to get a partial coloring $c'$, i.e., $c'(u,i)=c(u,j)$, $c'(u,j)=c(u,i)$ and $c'(e)=c(e)$ for all other edges $e$ in $G$. The color exchange with respect to the edges $ui$ and $uj$ is said to be proper if the coloring obtained after the exchange is proper. The color exchange with respect to the edges $ui$ and $uj$ is $valid$ if and only if the coloring obtained after the exchange is acyclic. The following fact is obvious:

\begin{fact}
\label{fact:fact3}
Let $c'$ be the partial coloring obtained from a valid partial coloring $c$ by the color exchange with respect to the edges $ui$ and $uj$. Then the partial coloring $c'$ will be proper if and only if $c(u,i) \notin S_{uj}$ and $c(u,j) \notin S_{ui}$.
\end{fact}

~~~~~~

\noindent The color exchange is useful in breaking some critical paths as is clear from the following lemma:

\begin{lem}
\label{lem:lem1}
Let $u,i,j,a,b \in V(G)$, $ui,uj,ab \in E$. Also let $\{\lambda,\xi\} \in C$ such that $\{\lambda,\xi\} \cap \{c(u,i),c(u,j)\} \neq \emptyset$ and $\{i,j\} \cap \{a,b\} = \emptyset$. Suppose there exists an ($\lambda$,$\xi$,$ab$)-critical path that contains vertex $u$, with respect to a valid partial coloring $c$ of $G$. Let $c'$ be the partial coloring obtained from $c$ by the color exchange with respect to the edges $ui$ and $uj$. If $c'$ is proper, then there will not be any ($\lambda$,$\xi$,$ab$)-critical path in $G$ with respect to the partial coloring $c'$.
\end{lem}
\begin{proof}
Firstly, $\{\lambda,\xi\} \neq \{c(u,i),c(u,j)\}$. This is because, if there is a ($\lambda$,$\xi$,$ab$)-critical path that contains vertex $u$, with respect to a valid partial coloring $c$ of $G$, then it has to contain the edge $ui$ and $uj$. Since $i \notin \{a,b\}$, vertex $i$ is an internal vertex of the critical path which implies that both the colors $\lambda$ and $\xi$ (that is $c(u,i)$ and $c(u,j)$) are present at vertex $i$. That means $c(u,j) \in S_{ui}$ and this contradicts $Fact$ \ref{fact:fact3}, since we are assuming that the color exchange is proper. Thus $\{\lambda,\xi\} \neq \{c(u,i),c(u,j)\}$.

Now let $P$ be the ($\lambda$,$\xi$,$ab$) critical path with respect to the coloring $c$. Without loss of generality assume that $\gamma = c(u,i) \in \{\lambda,\xi\}$. Since vertex $u$ is contained in path $P$, by the maximality of the path $P$, it should contain the edge $ui$ since $c(u,i)=\gamma \in \{\lambda,\xi\}$. Let us assume without loss of generality that path $P$ starts at vertex $a$ and reaches vertex $i$ before it reaches vertex $u$. Now after the color exchange with respect to the edges $ui$ and $uj$, i.e., with respect to the coloring $c'$, there will not be any edge adjacent to vertex $i$ that is colored $\gamma$. So if any ($\lambda$,$\xi$) maximal bichromatic path starts at vertex $a$, then it has to end at vertex $i$. Since $i \neq b$, by $Fact$ \ref{fact:fact1} we infer that the ($\lambda$,$\xi$,$ab$) critical path does not exist.
\renewcommand{\qedsymbol}{\bqed}
\end{proof}

\section{Proof of Theorem 1}

\begin{proof}
We prove the theorem by way of contradiction. Let $G$ be a 2-degenerate graph with $n$ vertices and $m$ edges which is a minimum counter example for the theorem statement. Then the theorem is true for all 2-degenerate graphs with at most $m-1$ edges. To prove the theorem for $G$, we may assume that $G$ is connected. We may also assume that the minimum degree, $\delta(G)\ge 2$, since otherwise if there is a vertex $v$, with $degree(v)= 1$, we can easily extend the acyclic edge coloring of $G-e$ (where e is the edge incident on $v$) to $G$. Keeping the assumption that $G$ is a minimum counter example in mind we will show that any partial coloring $c$ of $G$ should satisfy certain properties which in turn will lead to a contradiction.

~~~~~

\textbf{Selection of the Primary Pivot:} Let $W_0 = \{z \in V(G)\ \vert\ degree_{G}(z)=2\}$. Since $G$ is 2-degenerate $W_0 \neq \emptyset$. We may assume that $V-W_0 \neq \emptyset$ because otherwise, $G$ is a cycle and it is easy to see that it is $\Delta+1 = 3$ acyclically edge colorable. Thus $\Delta(G) \ge 3$. Let $G'=G[V-W_0]$ and $W_1 = \{z \in V(G')\ \vert\ degree_{G'}(z) \le 2\}$. By the definition of 2-degeneracy there exists at least one vertex of degree at most 2 in $G'$ and thus $W_1 \neq \emptyset$.

Let $V' = V(G')$. If $V'-W_1 \neq \emptyset$, then there exists at least one vertex of degree at most 2 in $G'[V'-W_1]$. Let $G''=G[V'-W_1]$ and $W_2 = \{z \in V(G'')\ \vert\ degree_{G''}(z) \le 2\}$. Let $q \in W_2$. Clearly $N_{G}(q) \cap W_1 \neq \emptyset$ and let $x \in N_{G}(q) \cap W_1$. On the other hand if $V'-W_1 = \emptyset$, then let $x \in W_1$. We call $x$ the $Primary\ Pivot$, since $x$ plays an important role in our proof. Let $N'_{G}(x) = N_{G}(x) \cap W_0$ and $N''_{G}(x) = N_{G}(x) - N'_{G}(x)$ = $N_{G'}(x)$. Since $x \in W_1$, it is easy to see that $\vert N''_{G}(x) \vert \le 2$ and $\vert N'_{G}(x) \vert \ge 1$.

~~~~~~

Let $N'_{G}(x) = \{y_1,\ y_2,\ldots, y_t\}$. Also $\forall y_i$, let $N_{G}(y_i)=\{x,y'_i\}$ (See $figure$ \ref{fig:fig1}). $\forall y_i$, let $G_i$ denote the graph obtained by removing the edge $(x,y_i)$ from the graph $G$. Let $N'_{G_i}(x) = N'_{G}(x) - \{y_i\}$ and $N''_{G_i}(x) = N_{G_i}(x) - N'_{G_i}(x)$. By induction on the number of edges, graph $G_i$ is $\Delta+1$ acyclically edge colorable. Let $c_i$ be a valid coloring of $G_i$ and thus a partial coloring of $G$. We denote the set of colors by $C = \{1,2,\ldots,\Delta+1\}$.

~~~~~

\begin{figure}[!h]
\begin{center}
\includegraphics[width= 80 mm, height = 60 mm]{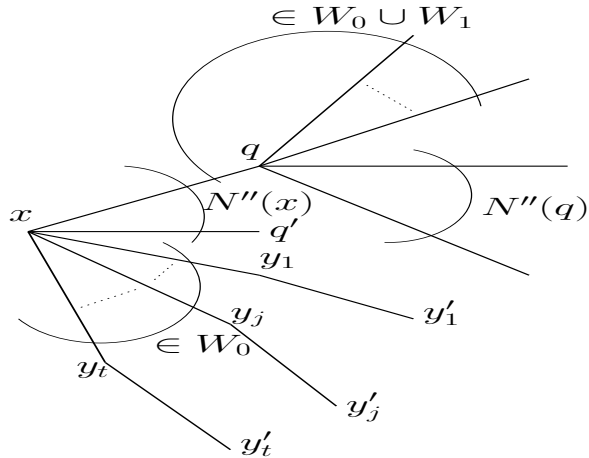}
\caption{Vertex $x$ and its neighbours}
\label{fig:fig1}
\end{center}
\end{figure}

\textbf{Comment:} Note that the figures given in this paper are only for providing visual aid for the reader. They do not capture all possible configurations.

~~~~~~~~

\subsection{Properties of any valid coloring $c_i$ of $G_i$}

~~~~~~~

Let $F_x(c_i)=\{c_i(x,z) \vert z \in N_{G_i}(x)\}$. Let $F'_{x}(c_i) =\{c_i(x,z) \vert z \in N'_{G_i}(x)\}$ and $F''_{x}(c_i) =\{c_i(x,z) \vert z \in N''_{G_i}(x)\}$. Note that $F_{x}(c_i)$ is the disjoint union of $F'_{x}(c_i)$ and $F''_{x}(c_i)$ and also $\vert F''_{x}(c_i) \vert \le 2$.

~~~~~

\begin{lem}
\label{lem:lem2}
With respect to any valid coloring $c_i$ of $G_i$, $c_i(y_i,y'_i) \in F''_{x}(c_i)$.
\end{lem}
\begin{proof}
\renewcommand{\qedsymbol}{\bqed}
It is easy to see that $c_i(y_i,y'_i) \in F_{x}(c_i)$. Otherwise all the candidate colors are valid for the edge $xy_i$, since any cycle involving the edge $xy_i$ will contain  the edge $y_iy'_i$ as well as an edge incident on $x$ in $G_i$ and thus the cycle will have at least 3 colors. Now if $c_i(y_i,y'_i) \in  F'_{x}(c_i)$ we have $\vert F_{x}(c_i) \cup \{c_i(y_i,y'_i)\}\vert \le \Delta-1$. Thus there are at least two \emph{candidate colors} for the edge $xy_i$. Let $y_j \in N'_{G_i}(x)$ be the vertex such that $c_i(y_i,y'_i)=c_i(x,y_j)$. When we color edge $xy_i$ there is a possibility of a bichromatic cycle only if we assign $c_i(y_j,y'_j)$ to the edge $xy_i$ since $degree_{G_i}(y_j) = 2$. But since we have at least two \emph{candidate colors} for edge $xy_i$, this situation can easily be avoided. We infer that $c_i(y_i,y'_i) \in F''_{x}(c_i)$.
\end{proof}

~~~~~

\begin{lem}
\label{lem:lem3}
With respect to any valid coloring $c_i$ of $G_i$, $\vert F''_{x}(c_i) \vert = 2$
\end{lem}
\begin{proof}
\renewcommand{\qedsymbol}{\bqed}
Suppose not. Then $\vert F''_{x}(c_i)\vert \le 1$. Since $ \vert F_{y'_i}(c_i) \vert \le \Delta$, we have at least one \emph{candidate color} for the edge $y_iy'_i$. Note that any \emph{candidate color}, is $valid$ for the edge $y_iy'_i$ in $G_i$ since $y_i$ is a pendant vertex in $G_i$. Let $c'_i$ be the valid coloring obtained by recoloring the edge $y_iy'_i$ with a candidate color. By $Lemma$ \ref{lem:lem2}, we have $c_i(y_i,y'_i) \in F''_{x}(c_i)$. Clearly since $\vert F''_{x}(c_i)\vert \le 1$ and $c'_i(y_i,y'_i) \neq c_i(y_i,y'_i)$, we can infer that $c'_i(y_i,y'_i) \notin F''_{x}(c_i)$, a contradiction to $Lemma$ \ref{lem:lem2}.
\end{proof}

~~~~~

\noindent An immediate consequence of $Lemma$ \ref{lem:lem3} is that $\vert N''_{G}(x) \vert = 2$. Moreover by the way we have selected vertex $x$ at least one of them should belong to $W_1 \cup W_2$. We make the following assumption:

\begin{asm}
\label{asm:asm1}
With respect to any valid coloring $c_i$ of $G_i$, without loss of generality let $F''_{x}(c_i) = \{1,2\}$ and $N''_{G}(x)=\{q,q'\}$. Thus $\{c_i(x,q),c_i(x,q')\}= \{1,2\}$. Also without loss of generality we assume that $q \in W_2 \cup W_1$ (see $figure$ \ref{fig:fig1}).
\end{asm}

~~~~~~

\begin{lem}
\label{lem:lem4}
With respect to any valid coloring $c_i$ of $G_i$, colors $1,\ 2 \notin S_{y_iy'_i}$.
\end{lem}
\begin{proof}
\renewcommand{\qedsymbol}{\bqed}
Since $ \vert F_{y'_i}(c_i) \vert \le \Delta$, we have at least one \emph{candidate color} $\gamma \neq c_i(y_i,y'_i)$ for the edge $y_iy'_i$. Note that $\gamma$ is $valid$ for the edge $y_iy'_i$ in $G_i$ since $y_i$ is a pendant vertex in $G_i$. Let $c'_i$ be the valid coloring obtained by recoloring the edge $y_iy'_i$ with $\gamma$. Now since $c_i$ as well as $c'_i$ are valid, by $Lemma$ \ref{lem:lem2}, we have $\{c_i(y_i,y'_i),c'_i(y_i,y'_i)\}=F''_{x}(c_i)=\{1,2\}$ (by $Assumption$ \ref{asm:asm1}). Since $c_i(y_i,y'_i) \notin S_{y_iy'_i}$ and $c'_i(y_i,y'_i) \notin S_{y_iy'_i}$, we have $1,\ 2 \notin S_{y_iy'_i}$.
\end{proof}

~~~~~~~~~

\noindent Let $C' = C- \{1,2\}$. For each color $\gamma \in C'$, we define a graph $G_{i,\gamma}$ as below:
\begin{displaymath}
G_{i,\gamma} = \left\{ \begin{array}{ll}
 G_i & \textrm{if $\gamma \in C'-F'_{x}(c_i)$}\\
 G_i-{xy_a},\ where\ c_i(x,y_a)=\gamma & \textrm{if $\gamma \in F'_{x}(c_i)$}
  \end{array} \right.
\end{displaymath}

\noindent Also let $c_{i,\gamma}$ be the valid coloring of $G_{i,\gamma}$ derived from $c_i$ of $G_i$, that is by discarding the color of the edge $xy_a$, where $y_a$ is the vertex such that $c_i(x,y_a)=\gamma$. Also if $c_{i,\gamma}$ is a valid coloring of $G_{i,\gamma}$, then $c_{i,\gamma}$ is said to be derivable from $c_1$ if we can extend the coloring $c_{i,\gamma}$ of $G_{i,\gamma}$ to the coloring $c_1$ of $G_1$.

~~~~~~~

\begin{lem}
\label{lem:lem5}
Let $c_i$ be any valid coloring of $G_i$. With respect to coloring $c_{i,\gamma}$ of $G_{i,\gamma}$, $\forall \gamma \in C'-F'_{x}(c_i)$, $\exists (\mu,\gamma,xy_i)$ critical path, where $\mu=c_i(y_i,y'_i)$.
\end{lem}
\begin{proof}
\renewcommand{\qedsymbol}{\bqed}
Recall that when $\gamma \in C'-F'_{x}(c_i)$, we have $G_{i,\gamma}=G_i$ and hence $c_{i,\gamma}=c_i$. Suppose if there is no $(\mu,\gamma,xy_i)$ critical path, where $\gamma \in C'-F'_{x}(c_i)$, then by $Fact$ \ref{fact:fact2} color $\gamma$ is valid for the edge $xy_i$. Thus we get a valid coloring of $G$, a contradiction.
\end{proof}

\begin{lem}
\label{lem:lem6}
Let $c_i$ be any valid coloring of $G_i$. With respect to coloring $c_{i,\gamma}$ of $G_{i,\gamma}$, $\forall \gamma \in C'-F'_{x}(c_i)$, $\exists (\nu,\gamma,x,y'_i)$ maximal bichromatic path, where $\{\nu\}=\{1,2\}-\{c_i(y_i,y'_i)\}$.
\end{lem}
\begin{proof}
\renewcommand{\qedsymbol}{\bqed}
Recall that when $\gamma \in C'-F'_{x}(c_i)$, we have $G_{i,\gamma}=G_i$ and hence $c_{i,\gamma}=c_i$. Suppose if there is no $(\nu,\gamma,x,y'_i)$ maximal bichromatic path, where $\gamma \in C'-F'_{x}(c_i)$, then by $Lemma$ \ref{lem:lem4}, color $\nu$ is a candidate for the edge $y_iy'_i$. Now recolor the edge $y_iy'_i$ with color $\nu$ to get a valid coloring $c'_{i}$ of $G_{i}$. Since by our assumption that there is no $(\nu,\gamma,x,y'_i)$ maximal bichromatic path with respect to $c_{i,\gamma}=c_i$, there cannot be any $(\nu,\gamma,xy_i)$ critical path with respect to the coloring $c'_{i}$, a contradiction to $Lemma$ \ref{lem:lem5} (Note that the color $\mu$ discussed in Lemma \ref{lem:lem5} and assumption is same as $\nu=c'_i(y_i,y'_i)$ in $c'_i$).
\end{proof}

~~~~~

\begin{asm}
\label{asm:asm21}
Since $\vert F_{x}(c_i)  \vert \le \Delta -1$, we have $\vert C-F_{x}(c_i) \vert \ge 2$. Since $C-F_{x}(c_i)= C'-F'_{x}(c_i)$, we have $\vert C'-F'_{x}(c_i) \vert \ge 2$. Thus $degree_{G_i}(y'_i) \ge 3$ and hence $degree_{G}(y'_i) \ge 3$. Let $\alpha,\beta \in C'-F'_{x}(c_i)$.
\end{asm}

~~~~~

\begin{lem}
\label{lem:lem7}
Let $c_i$ be any valid coloring of $G_i$. With respect to coloring $c_{i,\gamma}$ of $G_{i,\gamma}$, $\forall \gamma \in F'_{x}(c_i)$, $\exists (\mu,\gamma,xy_i)$ critical path, where $\mu=c_i(y_i,y'_i)$.
\end{lem}
\begin{proof}
Let $c_i(x,y_j)= \gamma$, where $\gamma \in F'_{x}(c_i)$. Suppose if there is no $(\mu,\gamma,xy_i)$ critical path, then by $Fact$ \ref{fact:fact2} color $\gamma$ is valid for the edge $xy_i$ with respect to the coloring $c_{i,\gamma}$. Color the edge $xy_i$ with color $\gamma$ to get a valid coloring $d$ of $G-\{xy_j\}$.

\noindent Now we will show that we can extend the coloring $d$ of $G-\{xy_j\}$ to a valid coloring of the graph $G$ by giving a valid color for the edge $xy_j$, leading to a contradiction of our assumption that $G$ was a minimum counter example. We claim the following:

~~~~~

\begin{clm}
\label{clm:clm1}
With respect to the coloring $d$, either color $\alpha$ or $\beta$ is valid for the edge $xy_j$ (Recall that $\alpha , \beta \in C' - F'_{x}(c_i)$ by $Assumption \ref{asm:asm21}$)
\end{clm}
\begin{proof}
Without loss of generality, let $d(y_j,y'_j)=\eta$. Note that $\eta \neq \gamma = c_i(x,y_j)$. Now if,
\begin{enumerate}
\item $\eta \notin F_{x}(c_i)$. In view of Assumption \ref{asm:asm21}, $\alpha$, $\beta$ $\notin F_{x}(c_i)$. Noting that $\eta$ cannot be equal to both $\alpha$ and $\beta$, without loss of generality, let $\eta \neq \alpha$. Then color the edge $(x,y_j)$ with color $\alpha$ to get a proper coloring $d'$. If a bichromatic cycle gets formed, then it should contain the edge $xy_j$ and also involve both the colors $\eta$ and $\alpha$ since $degree_{G}(y_j)=2$. But since $\eta \notin F_{x}(c_i)$, such a bichromatic cycle is not possible. Thus the coloring $d'$ is valid.

\item $\eta \in \{1,2\}=\{\mu,\nu\}=F''_{x}(c_i)$. Recolor the edge $xy_j$ with color $\alpha$ to get a coloring $d'$. We claim that the coloring $d'$ is valid. This is because if it is not valid, then there has to be a $(\alpha,\eta)$ bichromatic cycle containing the edge $xy_j$ with respect to $d'$. This implies that there has to be a $(\eta,\alpha,xy_j)$ critical path with respect to the coloring $d$ and hence with respect to the coloring $c_{i,\gamma}$ (Note that the coloring $d$ is obtained from $c_{i,\gamma}$ just by giving the color $\gamma$ to the edge $xy_i$ and $\eta, \gamma \neq \alpha,\beta$).

If $\eta= \mu$, this means that there was a $(\eta=\mu,\alpha,xy_j)$ critical path with respect to $c_{i,\gamma}$. But this is not possible by $Fact$ \ref{fact:fact1} since there is already a $(\mu,\alpha,xy_i)$ critical path with respect to $c_{i,\gamma}$ (by $Lemma$ \ref{lem:lem5}) and $y_i \neq y_j$.

Thus $\eta = \nu$. This means that there has to be a $(\eta=\nu,\alpha,xy_j)$ critical path with respect to $c_{i,\gamma}$. But this is not possible by $Fact$ \ref{fact:fact1} since there is already a $(\nu,\alpha,x,y'_i)$ maximal bichromatic path with respect to $c_{i,\gamma}$ (by $Lemma$ \ref{lem:lem6}) and $y'_i \neq y_j$ ( $y'_i \neq y_j$ since by $Assumption$ \ref{asm:asm21}, $degree_{G_i}(y'_i) \ge 3$. But $degree_{G_i}(y_j) = 2$). Thus there cannot be any bichromatic cycles with respect to the coloring $d'$. Thus the coloring $d'$ is valid.

\item $\eta \in F'_{x}(c_i)$. Let $y_k \in N'_{G}(x)$ be such that $d(x,y_k)=\eta$. With respect to colors $\{\alpha,\beta\}$, without loss of generality let $d(y_k,y'_k) \neq \beta$. Recall that $d(y_j,y'_j)=\eta$. Now recolor the edge $xy_j$ with color $\beta$ to get a coloring $d'$. Now if a bichromatic cycle gets formed, then it should contain the edge $xy_j$ and also involve both the colors $\eta$ and $\beta$. Thus the bichromatic cycle should contain the edge $xy_k$. Since $degree_{G}(y_k)=2$, the bichromatic cycle should contain the edge $y_ky'_k$. But by our assumption, $c_i(y_k,y'_k) \neq \beta$, a contradiction. Thus the coloring $d'$ is valid.
\end{enumerate}

Hence either color $\alpha$ or $\beta$ is valid for the edge $xy_j$.

\end{proof}
\noindent Thus we have a valid coloring (i.e, $d'$) for the graph $G$, a contradiction.
\renewcommand{\qedsymbol}{\bqed}
\end{proof}

~~~~

\begin{lem}
\label{lem:lem8}
Let $c_i$ be any valid coloring of $G_i$. With respect to coloring $c_{i,\gamma}$ of $G_{i,\gamma}$, $\forall \gamma \in F'_{x}(c_i)$, $\exists (\nu,\gamma,x,y'_i)$ maximal bichromatic path, where $\{\nu\}=\{1,2\}-\{c_i(y_i,y'_i)\}$.
\end{lem}
\begin{proof}
Suppose if there is no $(\nu,\gamma,x,y'_i)$ maximal bichromatic path, where $\gamma \in F'_{x}(c_i)$, then by $Lemma$ \ref{lem:lem4}, color $\nu$ is a candidate for the edge $y_iy'_i$. Now recolor the edge $y_iy'_i$ with color $\nu$ to get a valid coloring $c'_{i,\gamma}$ of $G_{i}$. Since by our assumption that there is no $(\nu,\gamma,x,y'_i)$ maximal bichromatic path with respect to $c_{i,\gamma}$, there cannot be any $(\nu,\gamma,xy_i)$ critical path with respect to the coloring $c'_{i,\gamma}$, a contradiction to $Lemma$ \ref{lem:lem7} (Note that the color $\mu$ discussed in Lemma \ref{lem:lem7} and assumption is same as $\nu=c'_{i,\gamma}(y_i,y'_i)$ in $c'_{i,\gamma}$).

\renewcommand{\qedsymbol}{\bqed}
\end{proof}

~~~~~

\noindent \textbf{Critical Path Property:}In the rest of the paper we will have to repeatedly use the properties (namely the presence of $(\mu,\gamma,xy_i)$ critical path in $G_{i,\gamma}$, where $\mu=c_{i,\gamma}(y_i,y'_i)$) described by $Lemma$ \ref{lem:lem5} and $Lemma$ \ref{lem:lem7}. Therefore we will name these properties as the \emph{Critical Path Property} of the graph $G_{i,\gamma}$

~~~~~

\noindent If $c_i$ is any valid coloring of $G_i$, then in $G_{i,\gamma}$, $\forall \gamma \in C'$, by $Critical\ Path\ Property$ (i.e., $Lemma$ \ref{lem:lem5} or $Lemma$ \ref{lem:lem7}) there exists a $(\mu,\gamma,xy_i)$ critical path and by $Lemma$ \ref{lem:lem6} and $Lemma$ \ref{lem:lem8} there exists a $(\nu,\gamma,x,y'_i)$ maximal bichromatic path, where $\mu=c_i(y_i,y'_i)$ and $\{\nu\} = F''_{x}(c_i)- \{\mu\}$. Recall that $\vert S_{ab} \vert \le \Delta-1$ for any $ab \in E$. As an immidiate consequence we have,
\begin{eqnarray}
\label{eqn:eqn1}
 S_{xq} = S_{xq'}=S_{y_iy'_i}=C-\{1,2\}=C'.
\end{eqnarray}

\noindent In view of $(\ref{eqn:eqn1})$, we have
\begin{equation}
\label{eqn:eqn2}
\vert S_{xq} \vert = \vert S_{xq'} \vert = \vert S_{y_iy'_i} \vert = \vert C' \vert = \Delta -1.
\end{equation}

~~~~~~

\begin{lem}
\label{lem:lem9}
Let $c_i$ be any valid coloring of $G_i$. Let $\mu=c_i(y_i,y'_i) \in \{1,2\}$. Also let $y_j \in N'_{G}(x)-\{y_i\}$. Then $\forall \gamma \in C'$, the $(\mu,\gamma,xy_i)$ critical path in $G_{i,\gamma}$ does not contain the vertex $y_j$.
\end{lem}
\begin{proof}
Suppose there exists a $(\mu,\gamma,xy_i)$ critical path that contains the vertex $y_j$, then $y_j$ cannot be an end vertex as $y_i \neq y_j$. Thus $y_j$ is an internal vertex. Now since $degree_{G}(y_j)=2$, the $(\mu,\gamma,xy_i)$ critical path should contain the edge $xy_j$ as well. But the $(\mu,\gamma,xy_i)$ critical path ends at vertex $x$ with color $\mu$ which implies $c_i(x,y_j) =\mu$, a contradiction since $c_i(x,y_j) \notin \{1,2\}=\{\mu,\nu\}$.
\renewcommand{\qedsymbol}{\bqed}
\end{proof}

~~~~~~

\begin{lem}
\label{lem:lem10}
Let $c_i$ be any valid coloring of $G_i$ and let $u \in \{q,q'\}$ . Let $\mu=c_i(y_i,y'_i)= c_i(x,u)\in \{1,2\}$ and $\nu = \{1,2\}-\{\mu\}$. Then $\forall \gamma \in C'$, the $(\mu,\gamma,xy_i)$ critical path in $G_{i,\gamma}$ has length at least five.
\end{lem}
\begin{proof}
Suppose not. Then the $(\mu,\gamma,xy_i)$ critical path has length three which implies that the vertices in the critical path are $x$, $u$, $y'_i$, $y_i$ in that order. Thus $\mu \in F_{u}(c_i)$ and $F_{u}(c_i)=S_{xu} \cup \{\mu\}$. Now change the color of the edge $y'_iy_i$ to $\nu$. It is proper since by $\ref{eqn:eqn1}$, we have $\{1,2\}=\{\mu,\nu\} \notin S_{y_iy'_i}$. It is valid since $y_i$ is a pendant vertex in $G_{i,\gamma}$. Now in view of $Critical\ Path\ Property$ (i.e., $Lemma$ \ref{lem:lem5} or $Lemma$ \ref{lem:lem7}) there has to be a $(\nu,\gamma,xy_i)$ critical path that passes through the vertex $y'_i$ with respect to this new coloring. Since $c_{i,\gamma}(u,y'_i)=\gamma$, this $(\nu,\gamma,xy_i)$ critical path should contain vertex $u$ as an internal vertex, which implies that color $\nu \in F_{u}(c_i)$. Recalling that $F_{u}(c_i)=S_{xu} \cup \{\mu\}$, we have $\nu \in S_{xu}$, a contradiction in view of $(\ref{eqn:eqn1})$. Thus the $(\mu,\gamma,xy_i)$ critical path has length at least five with respect to the coloring $c_{i,\gamma}$ of $G_{i,\gamma}$.
\renewcommand{\qedsymbol}{\bqed}
\end{proof}

~~~~~

\subsection{The structure of the minimum counter example in the vicinity of the primary pivot, $x$}

\begin{lem}
\label{lem:lem11}
The minimum counter example $G$ satisfies the following properties,
\begin{enumerate}
\item[(a)] $\forall u,v \in N_{G}(x)$, $(u,v) \notin E(G)$.
\item[(b)] $\forall y_i \in N'_{G}(x)$ and $\forall v \in N_{G}(x)$, we have $(v,y'_i) \notin E(G)$.
\end{enumerate}
\end{lem}
\begin{proof}
To prove $(a)$ we consider the following cases:\newline
\noindent \emph{case 1.1: $u,v \in N'_{G}(x)$} \newline
Let $u=y_k$ and $v= y_j$. Now if $u \in N_{G}(v)$, then $u = y'_j$. Recalling that $\Delta(G) \ge 3$, in view of $(\ref{eqn:eqn2})$, we have $degree_G(u)=degree_G(y'_j) \ge 3$. But $degree_G(u)=degree_G(y_k) = 2$, a contradiction. \newline

\noindent \emph{case 1.2: $u,v \in N''_{G}(x)$} \newline
Then we need to show that $q' \notin N_{G}(q)$. To see this consider the coloring $c_i$ of graph $G_i$. We know that $\{c_i(x,q), c_i(x,q')\}=\{\mu,\nu\}$. Without loss of generality let $c_i(x,q)=c_i(y_i,y'_i)=\mu$. Note that by $(\ref{eqn:eqn2})$, we have $S_{xq} = C'$. If $q' \in N_{G}(q)$, then $c_i(q,q') \in C'$. Let $c_i(q,q')=\gamma \notin \{\mu,\nu\}$. Now in $G_{i,\gamma}$, the $(\mu,\gamma)$ maximal bichromatic path that starts at vertex $x$ contains only edges $xq$ and $qq'$ since $\mu \notin F_{q'}(c_i)$ (by $(\ref{eqn:eqn2})$). Thus by $Fact$ \ref{fact:fact1}, there cannot be a $(\mu,\gamma,xy_i)$ critical path in $G_{i,\gamma}$, a contradiction to $Critical\ Path\ Property$ (i.e., $Lemma$ \ref{lem:lem5} or $Lemma$ \ref{lem:lem7}). Thus $q' \notin N_{G}(q)$. \newline

\noindent \emph{case 1.3: $u \in N''_{G}(x)$ and $v \in N'_{G}(x)$} \newline
Let $v = y_i$. Then we have to show that $y'_i \notin N''_{G}(x)=\{q,q'\}$. To see this consider the coloring $c_i$ of graph $G_i$. Recall that $\{c_i(x,q), c_i(x,q')\}=\{\mu,\nu\}$. Without loss of generality let $c_i(x,q)=c_i(y_i,y'_i)=\mu$. Now if $y'_i = q$, then we have $c(q,y_i)=c_i(y'_i,y_i)=\mu$, a contradiciton since $c(x,q)=\mu$. On the other hand if $y'_i = q'$, then $c(q',y_i)=c_i(y'_i,y_i)=\mu$. This means that $\mu \in S_{xq'}$, a contradiction in view of $(\ref{eqn:eqn1})$. Thus $y'_i \neq q,q'$.
\newline

Thus $\forall u,v \in N_{G}(x)$, we have $(u,v) \notin E(G)$ \newline

\noindent To prove $(b)$ we consider the following cases:\newline
\noindent \emph{case 2.1: $v \in N'_{G}(x)$} \newline
Let $v= y_j \in N'_{G}(x)$. If $(v,y'_i)=(y_j,y'_i) \in E(G)$, then $y'_i = y'_j$. Consider the coloring $c_j$ of graph $G_j$. Let $c_j(y_j,y'_j)=\mu$. Recall that by $(\ref{eqn:eqn2})$, we have $S_{y_jy'_j} = C'$. If $y'_i = y'_j$, then $c_j(y'_j,y_i) \in C'$. Let $c_j(y'_j,y_i)=\gamma$. Now in $G_{j,\gamma}$, the $(\mu,\gamma)$ maximal bichromatic path that starts at vertex $y_j$ contains only edges $y_jy'_j$, $y'_jy_i$ and thus ends at vertex $y_i$ since $\mu \notin F_{y_i}(c_j)$. This is because $N_{G_{j,\gamma}}(y_i)=\{y'_j,x\}$ and we have $c_j(y'_j,y_i)=\gamma$ and $c_j(x,y_i) \neq \mu$( since by $Assumption$ \ref{asm:asm1}, $\mu \in c_j(x,q),c_j(x,q')\}$ ). Thus by $Fact$ \ref{fact:fact1}, there cannot be a $(\mu,\gamma,xy_j)$ critical path in $G_{j,\gamma}$, a contradiction to $Critical\ Path\ Property$ (i.e., $Lemma$ \ref{lem:lem5} or $Lemma$ \ref{lem:lem7}). Thus $y'_j \neq y'_i$. \newline

\noindent \emph{case 2.2: $v \in N''_{G}(x)=\{q,q'\}$} \newline
Then we have to show that $y'_i \notin N_{G}(q) \cup N_{G}(q')$. To see this consider the coloring $c_i$ of graph $G_i$. Recall that $\{c_i(x,q), c_i(x,q')\}=\{\mu,\nu\}$. Without loss of generality let $c_i(x,q)=c_i(y_i,y'_i)=\mu$. Suppose $y'_i \in N_{G}(q)$, then we have $c(y'_i,q) \in S_{xq}$. Thus by $(\ref{eqn:eqn1})$, we have $c_i(y'_i,q)\neq \nu$. Now there exists  a $(\mu,c_i(y'_i,q)\neq \nu,xy_i)$ critical path of length 3, a contradiction to $Lemma$ \ref{lem:lem10}. Now if $y'_i \in N_{G}(q')$, then we recolor the edge $y_iy'_i$ with color $\nu$ to get a valid coloring $c'_i$. Now there exists a $(\nu,c_i(y'_i,q'),xy_i)$ critical path of length 3, a contradiction to $Lemma$ \ref{lem:lem10}. Thus $y'_i \notin N_{G}(q) \cup N_{G}(q')$. \newline

Thus $\forall y_i \in N'_{G}(x)$ and $\forall v \in N_{G}(x)$, we have $(v,y'_i) \notin E(G)$.

\renewcommand{\qedsymbol}{\bqed}
\end{proof}

\subsection{Modification of valid coloring $c_1$ of $G_1$ to get valid coloring $c_j$ of $G_j$}

\begin{asm}
\label{asm:asm3}
Let $c_1$ be a valid coloring of $G_1$ and without loss of generality let $c_1(x,q)=1$, $c_1(x,q')=2$ and $c_1(y_1,y'_1)=\mu =1$.
\end{asm}

\noindent \textbf{Remark: }In view of $Assumption$ \ref{asm:asm3}, the $Critical\ Path\ Property$ with respect to the coloring $c_1$ of $G_1$ reads as follows: With respect to the coloring $c_{1,\gamma}$, there exists a $(1,\gamma,xy_1)$ critical path, for all $\gamma \in C'$ .

~~~~~~~~~~~~

Let $f_1$ be the coloring of $G_1$ obtained from $c_1$ by exchanging the colors of the edges $xq$ and $xq'$. Also for $\gamma \in C'$, we define the coloring $f_{1,\gamma}$ as the coloring obtained from $c_{1,\gamma}$ by exchanging the colors with respect to the edges $xq$ and $xq'$. Note that $f_{1,\gamma}$ can be obtained from $f_1$ just by discarding the $\gamma$ colored edge incident on vertex $x$ for $\gamma \in F'_{x}(f_1)$.

\begin{clm}
\label{clm:clm2}
The coloring $f_1$ is proper but is not valid.
\end{clm}
\begin{proof}
The coloring $f_1$ is proper since in view of $(\ref{eqn:eqn1})$, $2 \notin S_{xq}$ and $1 \notin S_{xq'}$. Suppose the coloring $f_1$ is valid. Let $\gamma$ be a candidate color for the edge $xy_1$. Clearly $\gamma \in C-F_{x}(f_1)$. Now since $f_1$ is proper, taking $u=x$, $i=q$, $j=q'$, $ab=xy_1$, $\lambda = 1$ and $\xi = \gamma$, $Lemma$ \ref{lem:lem1} can be applied. There existed a $(1,\gamma,xy_1)$ critical path with respect to coloring $c_1$. By $Lemma$ \ref{lem:lem1}, we infer that there cannot be any $(1,\gamma,xy_1)$ critical path with respect to the coloring $f_1$. Thus by Fact \ref{fact:fact2}, candidate color $\gamma$ is valid for the edge $xy_1$. Thus we have obtained a valid coloring for the minimum counter example $G$, a contradiction.
\end{proof}

By $Claim$ \ref{clm:clm2}, there exist bichromatic cycles with respect to the coloring $f_1$. It is clear that each bichromatic cycle with respect to $f_1$ has to contain either the edge $xq$ or $xq'$ since we have changed only the colors of the edges $xq$ and $xq'$ to get the coloring $f_1$ from $c_1$. Thus each such bichromatic cycle should be either a $(1,\gamma)$ bichromatic cycle or a $(2,\gamma)$ bichromatic cycle. Note that each of these bichromatic cycles should pass through the vertex $x$. Moreover observe that there cannot be any $(1,2)$ bichromatic cycle since color $1 \notin S_{xq}$ with respect to $f_1$ in view of $(\ref{eqn:eqn1})$. Thus $\gamma \in F'_{x}(f_1)$. From this we infer that $\vert F'_{x}(f_1) \vert \ge 1$. Recalling $Assumption$ \ref{asm:asm21}, we have $\vert C-F_{x}(f_1) \vert \ge 2$. It follows that $\vert C' \vert \ge 3$. Thus we have,
\begin{equation}
\label{eqn:eqn3}
 \Delta(G) \ge degree_{G_1}(q) \ge \vert S_{xq} \vert +1 \ge \vert C' \vert +1 \ge 4.
\end{equation}

Let
\[ \mbox{$C_1= C_1(f_1)=\{\gamma \in F'_{x}(c_1) \vert\ \exists (1,\gamma)$ bichromatic cycle with respect to coloring $f_1$\}.} \]
\[ \mbox{$C_2= C_2(f_1)=\{\gamma \in F'_{x}(c_1) \vert\ \exists (2,\gamma)$ bichromatic cycle with respect to coloring $f_1$\}.} \]

\noindent Note that from the discussion above, any bichromatic cycle with respect to the coloring $f_1$ contains a vertex $y_i \in N'_{G_1}(x)$. But $degree_{G_1}(y_i) = 2$ and therefore $\vert S_{xy_i} \vert = 1$. Thus $S_{xy_i}$ contains exactly one of the color 1 or 2. Thus with a fixed color $\gamma \in C_1 \cup C_2$ there exists exactly one of $(1,\gamma)$ or $(2,\gamma)$ bichromatic cycle, which implies that the sets $C_1$ and $C_2$ cannot have any element in common (See $figure$ \ref{fig:fig2}). Thus we have,
\begin{equation}
\label{eqn:eqn4}
 C_1 \cap C_2 = \emptyset.
\end{equation}

~~~~~~

\begin{figure}[!h]
\begin{center}
\includegraphics[width= 120 mm, height = 60 mm]{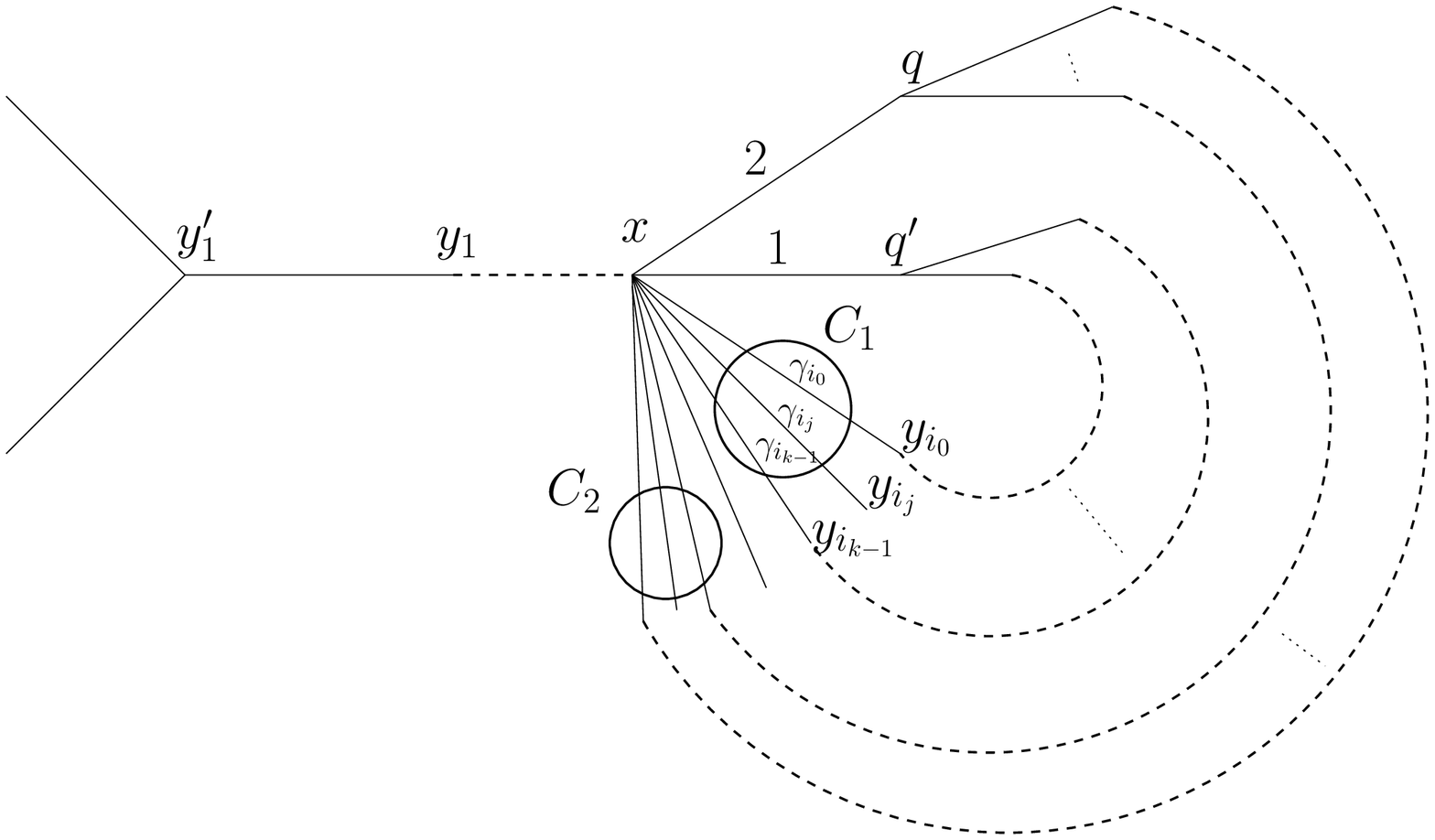}
\caption{Bichromatic cycles of $C_1$ and $C_2$}
\label{fig:fig2}
\end{center}
\end{figure}

~~~~~~

Recall that in view of $Critical\ Path\ Property$ (i.e., $Lemma$ \ref{lem:lem5} or $Lemma$ \ref{lem:lem7}), for a coloring $c_{1,\gamma}$ of $G_{1,\gamma}$, $\forall \gamma \in C'$, there exists a $(1,\gamma,xy_1)$ critical path. With respect to the new coloring $f_{1,\gamma}$, since the colors of only edges $xq$ and $xq'$ are changed, this path starts from $y_1$ and reaches the vertex $q$. But since color 1 is not present at vertex $q$ with respect to the coloring $f_{1,\gamma}$, the bichromatic path ends at vertex $q$. Thus the $(1,\gamma,xy_1)$ critical path with respect to coloring $c_{1,\gamma}$  gets curtailed to $(\gamma,1,q,y_1)$ maximal bichromatic path with respect to $f_{1,\gamma}$. Also note that in view of $Lemma$ \ref{lem:lem10} the length of this $(\gamma,1,q,y_1)$ maximal bichromatic path is at least four. This is true for the coloring $f_1$ also i.e., there exists a $(\gamma,1,q,y_1)$ maximal bichromatic path with respect to $f_{1}$. To see this observe that $f_1$ is obtained from $f_{1,\gamma}$ by putting back the edge $xy_a$, where $c_1(x,y_a)=\gamma$. This cannot alter the $(\gamma,1,q,y_1)$ maximal bichromatic path since $x$ does not belong to this path and also $y_a \neq y_1$. Also in view of Lemma \ref{lem:lem9}, none of the above maximal bichromatic paths contain vertex $y_j$, $\forall y_j \in N'_{G}(x)-\{y_1\}$. Thus the coloring $f_1$ satisfies the following property which we name as $Property\ A$:

~~~~~~

\noindent \textbf{Property A}:
A partial coloring of $G$ is said to satisfy $Property\ A$ iff $\forall \gamma \in C -\{1,2\}$, there exists a $(\gamma,1,q,y_1)$ maximal bichromatic path of length at least four. Moreover none of the above maximal bichromatic paths contain vertex $x$ or vertex $y_i$, where $y_i \in N'_{G}(x)-\{y_1\}$.

~~~~~~~

\begin{clm}
\label{clm:clm3}
There exists a proper coloring $f'_1$ obtained from $f_1$ such that $\forall i \in \{1,2\}$, $\vert C_i\vert \le 1$, where $C_i=C_i(f'_1)$. Moreover $f'_1$ satisfies $Property\ A$.
\end{clm}
\begin{proof}
If $\vert C_1\vert \le 1$ and $\vert C_2\vert \le 1$, then let $f'_1 = f_1$. If $\vert C_1\vert \le 1$, then let $f''_1 = f_1$. Otherwise if $\vert C_1\vert \ge 2$, then let $C_1 = \{\gamma_{i_0},\gamma_{i_1}, \ldots,\gamma_{i_{k-1}}\}$ and also let $y_{i_j}$ be the vertex such that $f_1(x,y_{i_j}) = \gamma_{i_j}$, $\forall j \in \{0,1,2,\ldots k-1\}$ (see $Figure$ \ref{fig:fig2}). Now let the coloring $f''_1$ be defined as $f''_1(x,y_{i_j}) = \gamma_{i_l}$, where $l = j+1 (mod\ k)$, $\forall j \in \{0,1,2,\ldots k-1\}$ and $f''_1(e)=f_1(e)$ for all other edges. (\emph{Note that we have only shifted the colors of the edges $xy_{i_0},xy_{i_1},\ldots,xy_{i_{k-1}}$ circularly. We call this procedure \bf{deranging of colors}.})

Note that we are changing only the colors of the edges $xy_{i_j}$ for $j = 0,1,2,\ldots,k-1$. Also we are using only the colors $\gamma_{i_j} \in C_1$ for recoloring. Since with respect to the coloring $f_1$, $(1,\gamma_{i_j})$ bichromatic cycle passed through $y_{i_j}$ and $degree_{G_1}(y_{i_j}) = 2$, we have $S_{xy_{i_j}} = \{1\}$. Thus the coloring $f''_1$ is proper.

Since for all $y_{i_j}$, $0 \le j\le k-1$ we have $S_{xy_{i_j}} = \{1\}$ with respect to the coloring $f''_1$, it is clear that any $new$ bichromatic cycle created (in the process of getting $f''_1$ from $f'_1$ ) has to be a $(1,\gamma)$ bichromatic cycles, where $\gamma \in C_1$.

We claim that the coloring $f''_1$ does not have any $(1,\gamma)$ bichromatic cycle for $\gamma \in C_1$. To see this consider a $\gamma \in C_1$, say $\gamma_{i_1}$. There existed a $(1,\gamma_{i_1})$ bichromatic cycle with respect to $f_1$. It contained the edge $xy_{i_1}$. Now with respect to $f''_1$ edge $xy_{i_1}$ is colored with color $\gamma_{i_2}$. Thus the $(1,\gamma_{i_1})$ maximal bichromatic path which contains the vertex $x$ has one end at vertex $y_{i_1}$ since color $\gamma_{i_1}$ is not present at the vertex $y_{i_1}$ with respect to $f''_1$. Thus $(1,\gamma_{i_1})$ bichromatic cycle cannot exist with respect to the coloring $f''_1$. This argument works for all $\gamma \in C_1$ and thus for any color $\gamma \in C_1$, there is no $(1,\gamma)$ bichromatic cycle with respect to $f''_1$.

If $\vert C_2\vert \le 1$, then $f'_1 = f''_1$. Otherwise if $\vert C_2\vert \ge 2$, by performing similar recoloring (now starting with $f''_1$) as we did to get rid of the $(1,\gamma)$ bichromatic cycles, we can get a coloring $f'''_1$ without any $(2,\gamma)$ bichromatic cycle. Now let $f'_1 = f'''_1$. Thus we get a coloring $f'_1$ from $f_1$ which has $\vert C_1\vert \le 1$ and $\vert C_2\vert \le 1$.

Note that we are changing only the colors of the edges $xy_i$, for $y_i \in N'_{G_1}(x)$. But the coloring $f_1$ satisfied $Property\ A$ and hence none of the $(\gamma,1,q,y_1)$ maximal bichromatic paths , $\forall \gamma \in C -\{1,2\}$, contained the vertex $y_i$ or $x$. Thus these bichromatic paths have not been altered (i.e., neither $broken$ nor $extended$) by the recoloring to get $f'_1$ from $f_1$. Thus the coloring $f'_1$ satisfies $Property\ A$.
\end{proof}

\begin{obs}
\label{obs:obs1}
Note that the color of the edge $y_1y'_1$ is unaltered in $f'_1$, i.e., $f'_1(y_1,y'_1)=f_1(y_1,y'_1)=c_1(y_1,y'_1)=1$. Also only the colors of certain edges incident on the vertex $y_i$, where $y_i \in N'_{G}(x)-\{y_1\}$ are modified when we obtained $f'_1$ starting from $c_1$. (\emph{This information is required later in the proof}).
\end{obs}

\noindent It is easy to see that $f'_1$ is proper but not valid. It is not valid because, if it is valid then since $f'_1$ satifies $Property\ A$, there are $(\gamma,1,q,y_1)$ maximal bichromatic paths , $\forall \gamma \in C -\{1,2\}$. Thus by $Fact$ \ref{fact:fact1}, for any $\theta \in C-F_{x}(f'_1)$, there cannot be a $(1,\theta,xy_1)$ critical path. Thus by $Fact$ \ref{fact:fact2}, color $\theta$ is valid for the edge $xy_1$. Thus we have a valid coloring for the graph G, a contradiciton. Thus $f'_1$ is not valid. It implies that at least one of $C_1$ or $C_2$ is nonempty. In the next lemma we further refine the proper coloring $f'_1$.

\begin{lem}
\label{lem:lem12}
There exists a proper coloring $h_1$ of $G_1$ obtained from $f'_1$ such that there is at most one bichromatic cycle. Moreover $h_1$ satisfies $Property\ A$.
\end{lem}
\begin{proof}
By $Claim$ \ref{clm:clm3}, we have $\vert C_1\vert \le 1$ and $\vert C_2\vert \le 1$. If exactly one of $C_1$, $C_2$ is singleton, then let $h_1=f'_1$. Otherwise we have $\vert C_1\vert = 1$ and $\vert C_2\vert = 1$.

\begin{asm}
\label{asm:asm4}
Without loss of generality let $C_1=\{\gamma\}$ and $C_2=\{\theta\}$. Let $f'_1(x,y_j)=\gamma$ and $f'_1(x,y_k)=\theta$. Thus $f'_1(y_j,y'_j)=1$ and $f'_1(y_k,y'_k)=2$, since there are $(1,\gamma)$ and $(2,\theta)$ bichromatic cycles passing through the vertex $x$.
\end{asm}

\begin{clm}
\label{clm:clm4}
Color $2 \notin S_{y_jy'_j}$ 
\end{clm}
\begin{proof}
Suppose not, then $2 \in S_{y_jy'_j}$. Since there is a $(1,\gamma)$ bichromatic cycle passing through $y'_j$, the colors 1  and $\gamma$ are present at $y'_j$. It follows that there exists $\eta \in C-\{1,2,\gamma\}$ missing at $y'_j$. Now recolor edge $y_jy'_j$ with color $\eta$ to get a coloring $f''_1$. If the color $\eta$ is valid for the edge $y_jy'_j$, then let $h_1 = f''_1$ and we are done as the situation reduces to having only one bichromatic cycle (i.e.,$\vert C_2\vert = 1$ and $\vert C_1\vert = 0$ ). If the color $\eta$ is not valid for the edge $y_jy'_j$, then there has to be a $(\gamma,\eta)$ bichromatic cycle that passes through vertex $x$. Let $f''_1(x,y_l)=\eta$. Since $degree_{G}(y_l)=2$, we have $S_{xy_l} = \{f''_1(y_l,y'_l)\}= \{\gamma\}$. Recall that by $Assumption$ \ref{asm:asm21}, $\alpha \in C'-F'_{x}(c_1)$ and thus $\alpha \in C'-F'_{x}(f''_1)$. Clearly $\alpha \neq \eta$.  Recolor the edge $xy_j$ with color $\alpha$ to get a coloring $f'''_1$. Note that the color $\alpha$ is valid for the edge $xy_j$ because if there is a $(\alpha,\eta)$ bichromatic cycle, then it implies that $S_{xy_l} = \{\alpha\}$. But we know that $S_{xy_l}=\{\gamma\}$, a contradiction. Thus let $h_1 = f'''_1$ and the situation reduces to having only one bichromatic cycle (i.e., $\vert C_2\vert = 1$ and $\vert C_1\vert = 0$)
\end{proof}

In view of $Claim$ \ref{clm:clm4}, color 2 is a candidate for the edge $y_jy'_j$. Recolor edge $y_jy'_j$ with color $2$ to get a coloring $f''_1$. If the color $2$ is valid for the edge $y_jy'_j$, then let $h_1=f''_1$ and the situation reduces to having only one bichromatic cycle (i.e., $\vert C_2\vert = 1$ and $\vert C_1\vert = 0$ ). If the color $2$ is not valid for the edge $y_jy'_j$, then there has to be a $(\gamma,2)$ bichromatic cycle created due to the recoloring, thereby reducing the situation to $\vert C_2\vert = 2$ and $\vert C_1\vert = 0$. Now we can recolor the graph using the procedure similar to that in the proof of $Claim$ \ref{clm:clm3} (i.e., derangement of colors in $C_2$) to get a valid coloring $h_1$ without any bichromatic cycles.

The coloring $f'_1$ satisfied $Property\ A$ and hence none of the $(\gamma,1,q,y_1)$ maximal bichromatic paths , $\forall \gamma \in C -\{1,2\}$, contained the vertex $y_j$. Thus none of the $(\gamma,1,q,y_1)$ maximal bichromatic paths will be $broken$ or $curtailed$ in the process of getting $h_1$ from $f'_1$. This is because we are changing only the colors of the edges incident on the vertex $y_j$ or $y_k$ and if a $(\gamma,1,q,y_1)$ maximal bichromatic path gets $broken$ or $curtailed$, it means that the vertex $y_j$ or $y_k$ was contained in those maximal bichromatic path, a contradiction to $Property\ A$ of $f'_1$ since $y_j \in N'_{G}(x)-\{y_1\}$. On the other hand, if any of these paths gets extended, then vertex $y'_j \in \{y_1,q\}$. But in view of $Lemma$ \ref{lem:lem11} (part $(a)$) this is not possible.
Thus the $(\gamma,1,q,y_1)$ maximal bichromatic paths have not been extended. Thus these bichromatic paths have not been altered by the recolorings to get $h_1$ from $f'_1$. Thus the coloring $h_1$ satisfies $Property\ A$.
\renewcommand{\qedsymbol}{\bqed}
\end{proof}

\begin{obs}
\label{obs:obs2}
Note that the color of the edge $y_1y'_1$ is unaltered in $h_1$, i.e., $h_1(y_1,y'_1)= f'_1(y_1,y'_1)=1$ ( by $Observation$ \ref{obs:obs1}). Also only the colors of certain edges incident on the vertex $y_i$, where $y_i \in N'_{G}(x)-\{y_1\}$ are modified.
\end{obs}

~~~~~~~~~~

\noindent It is easy to see that $h_1$ is proper but not valid. It is not valid because, if it is valid then since $h_1$ satifies $Property\ A$, there are $(\gamma,1,q,y_1)$ maximal bichromatic paths , $\forall \gamma \in C -\{1,2\}$. Thus by $Fact$ \ref{fact:fact1}, for any $\theta \in C-F_{x}(h_1)$, there cannot be a $(1,\theta,xy_1)$ critical path. Thus by $Fact$ \ref{fact:fact2}, color $\theta$ is valid for the edge $xy_1$. Thus we have a valid coloring for the graph G, a contradiciton. Thus $h_1$ is not valid. Then in view of $Lemma$ \ref{lem:lem12}, we make the following assumption:

\begin{asm}
\label{asm:asm5}
Without loss of generality let the only bichromatic cycle in the coloring $h_1$ of $G_1$ pass through the vertex $y_j$, $j \neq 1$. Also let $h_1(x,y_j)=\rho$.
\end{asm}

\noindent We get a coloring $c_j$ of $G_j$ from $h_1$ of $G_1$ by:
\begin{enumerate}
\item Removing the edge $xy_j$.
\item Adding the edge $xy_1$ and coloring it with the color $h_1(x,y_j)=\rho$.
\end{enumerate}

~~~~~~

Note that the coloring $c_j$ is proper since $\rho \neq c_j(y_1,y'_1)=h_1(y_1,y'_1)=1$ (by $Observation$ \ref{obs:obs2}) and $\rho \notin S_{y_1x}(c_j)$ (by the definition of $c_j$). Note that by removing the edge $xy_j$ we have broken the only bichromatic cycle that existed with respect to $h_1$. The coloring $c_j$ is valid because if there is a bichromatic cycle in $G_j$ with respect to $c_j$ then it should contain the edge $xy_1$ and thus it should be a $(1,\rho)$ bichromatic cycle since $c_j(x,y_1)=h_1(x,y_j)=\rho$ and $c_j(y_1,y'_1)=1$. Suppose there exists a $(1,\rho)$ bichromatic cycle in $G_j$ with respect to $c_j$, then by $Fact$ \ref{fact:fact2} there must have been a $(1,\rho,xy_1)$ critical path with respect to $h_1$. But by $Lemma$ \ref{lem:lem12}, the coloring $h_1$ satisfies $Property\ A$ and thus there was a $(\rho,1,q,y_1)$ maximal bichromatic path implying by $Fact$ \ref{fact:fact1} that there cannot be a $(1,\rho,xy_1)$ critical path with respect to $h_1$, a contradiction. It follows that the coloring $c_j$ of $G_j$ is acyclic. Therefore all the Lemmas in previous sections are applicable to the coloring $c_j$ also.

Now we may assume that $c_j(y_j,y'_j) =2$ because if $c_j(y_j,y'_j) =1$, then we can change the color of the edge $y_jy'_j$ to 2 without altering the validity of the coloring since $y_j$ is a pendant vertex in $G_j$. Thus we make the following assumption:

\begin{asm}
\label{asm:asm6}
Without loss of generality let $c_j(y_j,y'_j) =2$. Also recall that $c_j(x,q)=2$ and $c_j(x,q')=1$.
\end{asm}

\noindent \textbf{Remark: }In view of $Assumption$ \ref{asm:asm6}, the $Critical\ Path\ Property$ with respect to the coloring $c_j$ of $G_j$ reads as follows: With respect to the coloring $c_{j,\gamma}$, there exists a $(2,\gamma,xy_j)$ critical path, for all $\gamma \in C'$. The reader may contrast the $Critical\ Path\ Property$ of $c_j$ with that of $c_1$ (See remark after $Assumption$ \ref{asm:asm3}). This correspondence is very important for the proof.

~~~~

\begin{obs}
\label{obs:obs3}
Note that $c_j(x,q)=2$, $c_j(x,q')=1$, $c_j(y_1,y'_1)=1$, $c_j(y_j,y'_j)=2$ and $c_j(x,y_1)= \rho \notin \{1,2\}$. Also if $e$ is an edge such that none of its end points is $x$ or $y_i$, where $y_i \in N'_{G}(x)$, we have $c_j(e) = c_1(e)$.
\end{obs}

\begin{lem}
\label{lem:lem13}
Coloring $c_{j,\gamma}$ of $G_{j,\gamma}$ satisfies $Property\ A$.
\end{lem}
\begin{proof}
We consider the following cases:\newline
\noindent \emph{case 1: $\gamma \in C'-{\rho}$}\newline
Recall that the coloring $h_1$ satisfied $Property\ A$. In getting $c_j$ from $h_1$, we have only colored the edge $xy_1$ with color $\rho$ and have discarded the edge $xy_j$. Thus $\forall \gamma \in C'-\{\rho\}$, there exists a $(\gamma,1,q,y_1)$ maximal bichromatic path in $c_j$ also. Noting that by $Property\ A$, the maximal bichromatic path does not contain vertex $x$ or $y_i$, where $\forall y_i \in N'_{G}(x)-\{y_1\}$, we infer that even in $G_{j,\gamma}$ the $(\gamma,1,q,y_1)$ maximal bichromatic path is unaltered.

\noindent \emph{case 2: $\gamma =\rho$}\newline
Then $G_{j,\rho}$ is the graph obtained by removing the edge $xy_1$ from $G_j$ since $c_j(x,y_1)=\rho$. Recall that with respect to the coloring $h_1$ we have a $(\rho,1,q,y_1)$ maximal bichromatic path. Removal of edge $xy_1$ from $G_1$ cannot alter this path since $h_1$ satisfies $Property\ A$ and thus edge $xy_1$ is not in the path. Now the graph obtained is nothing but the graph $G_{j,\rho}$ with respect to the coloring $G_{j,\rho}$. Thus $G_{j,\rho}$ saitsfies $Property\ A$.
\renewcommand{\qedsymbol}{\bqed}
\end{proof}

\noindent \textbf{Property B}:
Let $c_{1,\eta}$ be a partial coloring of $G_{1,\eta}$, for $\eta \in C-\{1,2\}$. Then $c_{1,\eta}$ is said to satisfy $Property\ B$ iff $\forall \gamma \in C -\{1,2\}$, there exists a $(\gamma,2)$ maximal bichromatic path which starts at vertex $q$ and involves the vertex $y'_j$. Also the length of the segment of this bichromatic path between the vertices $q$ and $y'_j$ is at least three. Moreover in none of the above maximal bichromatic paths the segment between the vertices $q$ and $y'_j$ contains vertex $x$ or vertex $y_i$, where $y_i \in N'_{G}(x)$.

\begin{lem}
\label{lem:lem14}
Coloring $c_{1,\eta}$ of $G_{1,\eta}$ satisfies $Property\ B$, for $\eta \in C -\{1,2\}$.
\end{lem}
\begin{proof}
By $Critical\ Path\ Property$ (i.e., $Lemma$ \ref{lem:lem5} or $Lemma$ \ref{lem:lem7}) and $Lemma$ \ref{lem:lem10}, $\forall \gamma \in C'$, there exists a $(2,\gamma,x,y_j)$ critical path of length at least five in $G_{j,\gamma}$. Also by $Lemma$ \ref{lem:lem9}, these critical paths do not contain vertex $y_i$, $\forall y_i \in N'_{G}(x)-\{y_j\}$. Recall that we obtained $c_j$ from $c_1$ by a series of recolorings. How will the above mentioned critical paths change if we undo all these recolorings and get back $c_1$? Note that in the process of obtaining coloring $c_j$ from $c_1$, we have only changed the colors incident on the vertices $y_i$, where $y_i \in N'_{G}(x)$ and have exchanged the colors of the edges $xq$ and $xq'$ (by $Observation$ \ref{obs:obs3}). Thus only the colors of edge $xq$ and possiblly edge $y_jy'_j$ of these critical paths will get modified when we undo the recolorings. The reader may recall that the first step in getting $c_j$ from $c_1$ was to exchange the colors of edges $xq$ and $xq'$. It follows that with respect to a coloring $c_{1,\eta}$, there exists a $(\gamma,2)$ maximal bichromatic path which $starts\ at$ vertex $q$ and involves the vertex $y'_j$. It also follows that the length of the segment of the bichromatic path between the vertices $q$ and $y'_j$ is at least three. Moreover it is easy to see that none of the above maximal bichromatic paths the segment between the vertices $q$ and $y'_j$ contains vertex $x$ or vertex $y_i$, where $y_i \in N'_{G}(x)$.
\renewcommand{\qedsymbol}{\bqed}
\end{proof}

~~~~~

\subsection{Selection of secondary pivot $p$ and properties of $c_1$ and $c_j$ in the vicinity of $p$}

~~~~~~

Let $N'_{G}(q) = N_{G}(q) \cap (W_1 \cup W_0)$ and $N''_{G}(q) = N_{G}(q) - N'_{G}(q)$. Since $q \in W_2 \cup W_1$ (see $Assumption$ \ref{asm:asm1}) it is easy to see that $\vert N''_{G}(q) \vert \le 2$. Now recall that in view of $(\ref{eqn:eqn2})$ $degree_{G}(q)=\Delta$ and by $(\ref{eqn:eqn3})$, $\Delta \ge 4$. Thus we have $\vert N'_{G}(q) \vert \ge 2$.

Let $p \in N'_{G}(q)$ be such that $p \neq x$. In the rest of the proof, this vertex $p$ will play a central role. Therefore we name it as the $Secondary\ Pivot$. Let $c_1(q,p)=\eta$. Note that $\eta \in C'$ by $(\ref{eqn:eqn1})$. Thus by $Critical\ Path\ Property$ (i.e., $Lemma$ \ref{lem:lem5} or $Lemma$ \ref{lem:lem7}), there exists a $(1,\eta,xy_1)$ critical path with respect to the coloring $c_{1,\eta}$ that passes through the vertex $p$ and clearly $qp$ is the second edge of this critical path. Recalling that this critical path has length at least five (by $Lemma$ \ref{lem:lem10}), we can infer that $p \neq y_1$ and $degree_{G_1}(p)\ge 2$. Now since $p \in W_1 \cup W_0$, there is at most one neighbour of $p$ other than $q$ which is not in $W_0$. If such a vertex exists let it be $p'$. Otherwise clearly ($N_{G}(p) \cap W_0$)$\neq \emptyset$ and let $p' \in N_{G}(p) \cap W_0$. Thus $N_{G}(p)-\{q,p'\} \subseteq W_0$. If $N_{G}(p)-\{q,p'\} \neq \emptyset$, let $N_{G}(p)-\{q,p'\}= \{z_1,z_2,\ldots, z_k\}$. Also $\forall z_i$, let $N_{G}(z_i)=\{p,z'_i\}$ (See $figure$ \ref{fig:fig3}) (At this point the reader may note that the primary pivot $x$ and secondary pivot $p$ are somewhat structurally similar).

\begin{figure}[!h]
\begin{center}
\includegraphics[width= 120 mm,height = 60 mm]{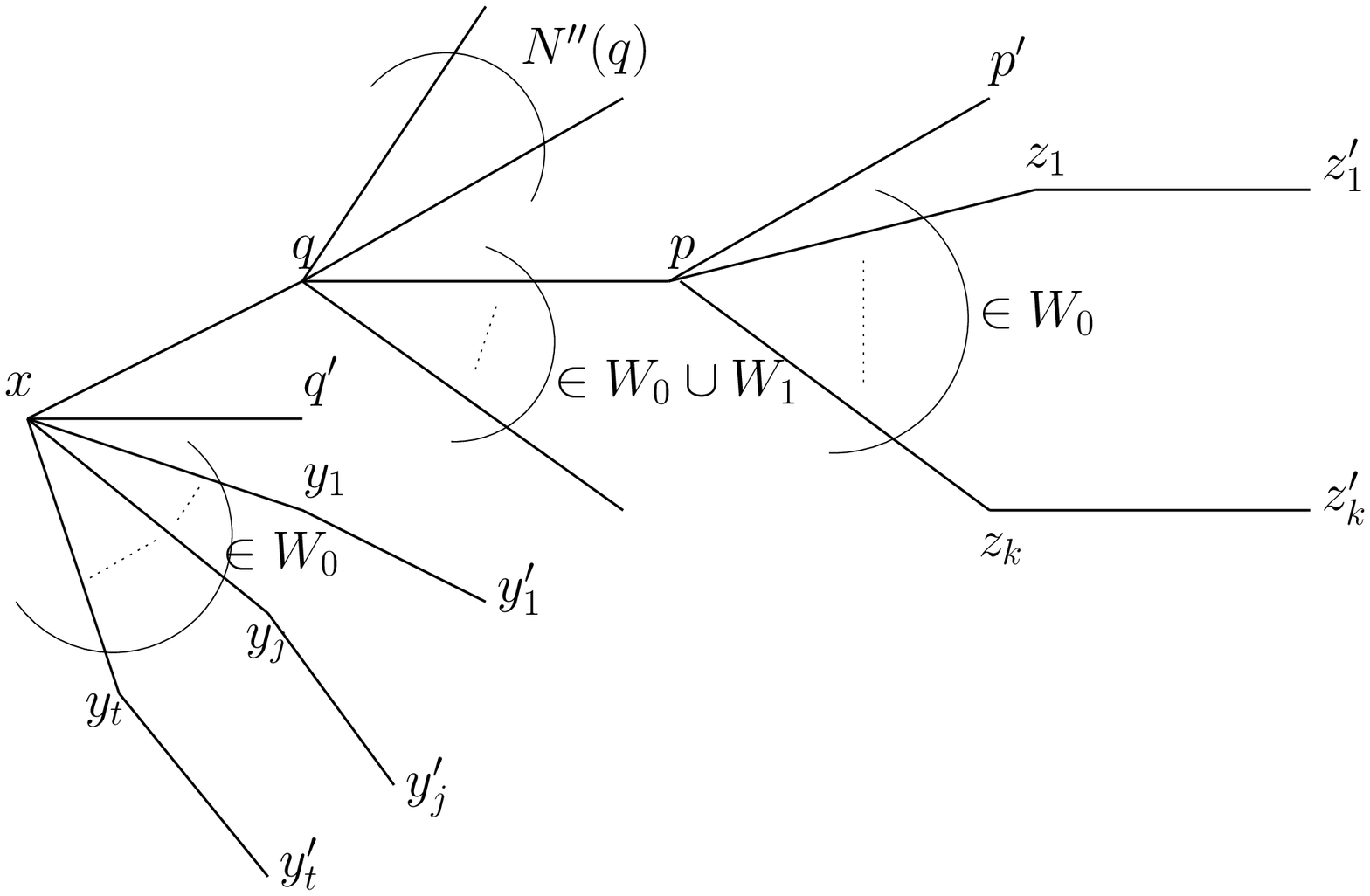}
\caption{Vertex $p$ and its neighbours}
\label{fig:fig3}
\end{center}
\end{figure}

\begin{lem}
\label{lem:lem15}
$x,y_i \notin \{p,p',z_1,\ldots,z_k,z'_1,\ldots,z'_k\}$, for $y_i \in N'_G(x)$.
\end{lem}
\begin{proof}
First note that $x \neq p$, by the definition of $p$. It is easy to see that $x \notin \{p',z_1,\ldots,z_k\}$, by part $(a)$ of $Lemma$ \ref{lem:lem11}. Now $x \notin \{z'_1,\ldots,z'_k\}$ because otherwise $z_i$ will be some $y_i$ and hence $p=y'_i$, But now there is an edge between $q$ and $p$, a contradiciton to part $(b)$ of $Lemma$ \ref{lem:lem11}. Similarly from part $(a)$ of $Lemma$ \ref{lem:lem11}, $y_i \neq p$ and from part $(b)$ of $Lemma$ \ref{lem:lem11}, $y_i \notin \{p',z_1,\ldots,z_k\}$. Now if $y_i \in \{z'_1,\ldots,z'_k\}$, then since $x \neq z_i$, we have $y'_i=z_i$, a contradiciton since $degree_{G}(y'_i)=\Delta \ge 4$ (by $(\ref{eqn:eqn2})$) and $degree_{G}(z_i)=2$. Thus $x,y_i \notin \{p,p',z_1,\ldots,z_k,z'_1,\ldots,z'_k\}$, for $y_i \in N'_G(x)$.
\renewcommand{\qedsymbol}{\bqed}
\end{proof}

\begin{lem}
\label{lem:lem16}
$c_j(q,p) = c_1(q,p)=\eta$.
\end{lem}
\begin{proof}
Recall that by $Observation$ \ref{obs:obs3}, only the edges incident on vertices $x$ or $y_i$, where $y_i \in N'_{G}(x)$ are altered while obtaining coloring $c_j$ from $c_1$. Now to show that $c_j(q,p) = c_1(q,p)=\eta$, its enough to verify that $q,p \notin \{x\} \cup N'_{G}(x)$. But this is obvious from part $(a)$ of $Lemma$ \ref{lem:lem11}.

\renewcommand{\qedsymbol}{\bqed}
\end{proof}

\begin{lem}
\label{lem:lem17}
$\{1,2\} \subseteq S_{qp}(c_{1,\eta})$.
\end{lem}
\begin{proof}
By $Lemma$ \ref{lem:lem10}, we know that $1 \in S_{qp}(c_{1,\eta})$ since $qp$ is only the second edge of the $(1,\eta,xy_1)$ critical path which is guaranteed to have length at least five with respect to the coloring $c_{1,\eta}$ of $G_{1,\eta}$. Now by $Lemma$ \ref{lem:lem14}, with respect to $c_{1,\eta}$ there exists a $(\eta,2)$ maximal bichromatic path which starts at vertex $q$ and contains vertex $y'_j$. Moreover the segment of the bichromatic path between the vertices $q$ and $y'_j$ is of length at least three with respect to $c_{1,\eta}$. Since this $(\eta,2)$ maximal bichromatic path starts with edge $qp$ colored $\eta$, we infer that $2 \in S_{qp}(c_{1,\eta})$. Thus $\{1,2\} \subseteq S_{qp}(c_{1,\eta})$.

\renewcommand{\qedsymbol}{\bqed}
\end{proof}

\noindent \textbf{Remark: }In view of $Lemma$ \ref{lem:lem17}, $degree_{G} \ge 3$. Therefore $p \notin W_0$. It follows that $p \in W_1$. It is interesting to note that $p$ could have been selected as the $Primary\ Pivot$ instead of $x$. The reader may want to reread the procedure for selecting the primary pivot given at the beginning of Section $3$. With respect to this procedure vertex $p$ is symmetric to vertex $x$ and thus is an equally eligible candidate to be the primary pivot. It follows that the structure of the minimum counter example at the vicinity of $p$ is symmetric to the structure at the vicinity of $x$. More specifically we have the following Lemma, corresponding to $Lemma$ \ref{lem:lem11}:

\begin{lem}
\label{lem:lem18}
The minimum counter example $G$ satisfies the following properties,
\begin{enumerate}
\item[(a)] $\forall u,v \in N_{G}(p)$, $(u,v) \notin E(G)$.
\item[(b)] $\forall z_i \in N_{G}(p)-\{q,p'\}$ and $\forall v \in N_{G}(p)$, we have $(v,z'_i) \notin E(G)$.
\end{enumerate}
\end{lem}

This Lemma is not explicitly used in the proof, but we believe that this information will help the reader to visualize the situation better.

~~~~~~

In view of $Lemma$ \ref{lem:lem17} let $e_1$ and $e_2$ be the edges incident on $p$ such that $c_{1,\eta}(e_1)=1$ and $c_{1,\eta}(e_2)=2$. Then we claim the following:

\begin{lem}
\label{lem:lem19}
$c_{j,\eta}(e_1)=1$ and $c_{j,\eta}(e_2)=2$.
\end{lem}
\begin{proof}
Recall that by $Observation$ \ref{obs:obs3}, only the edges incident on vertices $x$ or $y_i$, where $y_i \in N'_{G}(x)$ are altered while obtaining coloring $c_j$ from $c_1$. Let $e_1=(p,z_{i_1})$ and $e_2=(p,z_{i_2})$. Now to show that $c_{j,\eta}(e_1)=1$ and $c_{j,\eta}(e_2)=2$, it is enough to verify that $p,z_{i_1}),z_{i_2}) \notin \{x\} \cup N'_{G}(x)$. But this true by $Lemma$ \ref{lem:lem15}.

\renewcommand{\qedsymbol}{\bqed}
\end{proof}

\begin{lem}
\label{lem:lem20}
$c_{1,\eta}(p,p') \in \{1,2\}$ $($In other words, one of the edge $e_1$ or $e_2$ is $pp'$. By $Lemma$ \ref{lem:lem19}, this also implies that $c_{j,\eta}(p,p')=c_{1,\eta}(p,p') \in \{1,2\})$.
\end{lem}
\begin{proof}
Suppose not. Then $e_1 \neq pp'$ and $e_2 \neq pp'$. Without loss of generality let $e_1=(p,z_1)$ and $e_2=(p,z_2)$.  Thus $c_{1,\eta}(p,z_1)= 1$ and $c_{1,\eta}(p,z_2)= 2$. By $Lemma$ \ref{lem:lem10} there exists a $(1,\eta,xy_1)$ critical path of length at least five with respect to $c_{1,\eta}$. This implies that $c_{1,\eta}(z_1,z'_1)= \eta$. Now by $Lemma$ \ref{lem:lem14}, with respect to $c_{1,\eta}$ there exists a $(\eta,2)$ maximal bichromatic path which starts at vertex $q$ and contains vertex $y'_j$. Moreover the segment of this bichromatic path between the vertices $q$ and $y'_j$ is of length at least three with respect to $c_{1,\eta}$. Since $pz_1$ is only the second edge of this path, we can infer that $c_{1,\eta}(z_2,z'_2)= \eta$.

Now with respect to the coloring $c_{1,\eta}$, we exchange the colors of the edges $pz_1$ and $pz_2$ to get a coloring $c'_{1,\eta}$.

\begin{clm}
\label{clm:clm5}
Coloring $c'_{1,\eta}$ is valid.
\end{clm}
\begin{proof}
Note that $c'_{1,\eta}$ is proper since $c_{1,\eta}(z_1,z'_1)= \eta$ and $c_{1,\eta}(z_2,z'_2)= \eta$. Now the coloring $c'_{1,\eta}$ is valid because otherwise there has to be a $(\eta,1)$ or $(\eta,2)$ bichromatic cycle since only the colors of the edges $pz_1$ and $pz_2$ are altered. Thus such a bichromatic cycle has to contain the edge $qp$ since $c'_{1,\eta}(q,p)=\eta$. From $(\ref{eqn:eqn1})$, we can infer that color $2 \notin F_q(c'_{1,\eta})$. But if there exists a bichromatic cycle with respect to the coloring $c'_{1,\eta}$, it has to contain vertex $q$. From this we can infer that it has to be a $(\eta,1)$ bichromatic cycle. This means that the cycle has to contain the vertex $x$ since $c'_{1,\eta}(x,q)=1$. But we know by defintion of $c_{1,\eta}$ that $\eta \notin F_x(c_{1,\eta})=F_x(c'_{1,\eta})$. Thus there does not exist a $(\eta,1)$ bichromatic cycle with respect to the coloring $c'_{1,\eta}$. We conclude that the coloring $c'_{1,\eta}$ of $G_{1,\eta}$ is valid. 
\end{proof}

\begin{clm}
\label{clm:clm6}
With respect to the partial coloring $c'_{1,\eta}$, there does not exist any $(1,\eta,xy_1)$ critical path.
\end{clm}
\begin{proof}
Now since $c'_{1,\eta}$ is proper, taking $u=p$, $i=z_1$, $j=z_2$, $ab=xy_1$, $\lambda = 1$, $\xi = \eta$ and noting that $\{x,y_1\} \cap \{z_1,z_2\} = \emptyset$ (by $Lemma$ \ref{lem:lem15}), $Lemma$ \ref{lem:lem1} can be applied. There existed a $(1,\eta,xy_1)$ critical path containing vertex $p$ in coloring $c_{1,\eta}$. By $Lemma$ \ref{lem:lem1}, we infer that there cannot be any $(1,\eta,xy_1)$ critical path in the coloring $c'_{1,\eta}$.
\end{proof}

\begin{clm}
\label{clm:clm7}
There exists a valid coloring $c'_1$ of $G_1$ such that the coloring $c'_{1,\eta}$ of $G_{1,\eta}$ is derivable from $c'_1$.
\end{clm}
\begin{proof}
It is enough to show that we can extend the coloring $c'_{1,\eta}$ of $G_{1,\eta}$ to a valid coloring $c'_1$ of $G_1$. If $\eta \in C-F_x(c_1)$, then by definition $G_{1,\eta}=G_1$ and thus $c'_1=c'_{1,\eta}$. Otherwise let $y_k \in N'_{G}(x)$ be the vertex such that $c_1(x,y_k)=\eta$. Note that $k \neq 1$. Recall that $c_{1,\eta}$ is obtained by discarding the color on the edge $xy_k$. Thus it is enough to extend the coloring $c'_{1,\eta}$ to $c'_1$ by assigning an appropriate color to the edge $xy_k$.

Note that there exists a $(1,\alpha,xy_1)$ critical path with respect to $c_{1,\eta}$, for $\alpha \in C-F_x(c_1)$ (by $Lemma$ \ref{lem:lem5}). Clearly $\alpha \neq \eta$. We claim that the $(1,\alpha,xy_1)$ critical path exists even with respect to $c'_{1,\eta}$. To see this note that we have changed the colors of only edges $pz_1$ and $pz_2$ to get $c'_{1,\eta}$ from $c_{1,\eta}$. Note that by this exchange the $(1,\alpha,xy_1)$ critical path cannot be extended since $p,z_2 \notin \{x,y_1\}$ (by $Lemma$ \ref{lem:lem15}). Now if the $(1,\alpha,xy_1)$ critical path gets altered it means that this critical path contained the edge $pz_1$ (recall that $c_{1,\eta}(p,z_1)=1$) and hence $c_{1,\eta}(z_1,z'_1)=\alpha$. But we know that $c_{1,\eta}(z_1,z'_1)=\eta$, a contradiction. Thus we have,
\begin{eqnarray}
\label{eqn:eqn5}
\mbox{\emph{With respect to the partial coloring $c'_{1,\eta}$, there exists a $(1,\alpha,xy_1)$ critical path, for $\alpha \notin F_x(c'_{1,\eta})$ and $\alpha \neq \eta$.}}
\end{eqnarray}

Now color the edge $xy_k$ with color $\eta$ to get a coloring $d_1$ of $G_1$. If $d_1$ is valid we are done and $c'_1=d_1$. If it is not valid, then there has to be a bichromatic cycle containing the color $\eta$. Note that the coloring $d_1$ and $c_1$ differ only due to the exchange of colors of edges $pz_1$ and $pz_2$.  Thus it has contain one of the edges $pz_1$ or $pz_2$. Therefore it has to be either a $(\eta,1)$ or $(\eta,2)$ bichromatic cycle since $d_1(p,z_1)=2$, $d_1(p,z_2)=1$. This also means that the bichromatic cycle has to contain the vertex $q$, since $d(p,q)=\eta$. Thus the bichromatic cycle has to be a $(\eta,1)$ bichromatic cycle since $2 \notin F_q(d_1)$. This means that $d_1(y_k,y'_k)=1$. Now recolor the edge $xy_k$ with color $\alpha$ to get a coloring $d'_1$ of $G_1$. If $d'_1$ is valid we are done and $c'_1=d'_1$. If it is not valid then there has to be a $(\alpha,1)$ bichromatic cycle containing the vertex $x$, implying that there existed a $(1,\alpha,xy_k)$ critical path with respect to the coloring $d_1$ and hence with respect to the coloring $c'_{1,\eta}$. But in view of $(\ref{eqn:eqn5})$, there already exists a $(1,\alpha,xy_1)$ critical path and by $Fact$ \ref{fact:fact1}, $(1,\alpha,xy_k)$ critical path is not possible, a contradiction. Thus the coloring $d'_1$ is valid and let $c'_1=d'_1$.

Thus there exists a valid coloring $c'_1$ of $G_1$ such that the coloring $c'_{1,\eta}$ of $G_{1,\eta}$ is derivable from $c'_1$.
\end{proof}

\noindent Now in view of $Claim$ \ref{clm:clm6} and $Claim$ \ref{clm:clm7} there does not exists any $(1,\eta,xy_1)$ critical path with respect to the coloring $c'_{1,\eta}$ of $G_{1,\eta}$, a contradiction to $Critical\ Path\ Property$ (i.e., $Lemma$ \ref{lem:lem5} or $Lemma$ \ref{lem:lem7}).\newline

\noindent We conclude that $c_{1,\eta}(p,p') \in \{1,2\}$.

\renewcommand{\qedsymbol}{\bqed}
\end{proof}

\begin{asm}
\label{asm:asm7}
In view of $Lemma$ \ref{lem:lem17}, $Lemma$ \ref{lem:lem19} and $Lemma$ \ref{lem:lem20}, let $z_1$ be the vertex such that $\{c_{1,\eta}(p,z_1)\}= \{1,2\}-\{c_{1,\eta}(p,p')\}$. It follows that $\{c_{j,\eta}(p,z_1)\}=\{1,2\}-\{c_{j,\eta}(p,p')\}$ and $\{e_1,e_2\} =\{pp',pz_1\}$.
\end{asm}

\begin{obs}
\label{obs:obs4}
\begin{enumerate}
\item[(a)] If $c_{1,\eta}(p,p')=c_{j,\eta}(p,p')=2$, we have by Assumption \ref{asm:asm7} that $c_{1,\eta}(p,z_1)=c_{j,\eta}(p,z_1)=1$. Thus with respect to the partial coloring $c_{1,\eta}$, there exists a $(1,\eta,xy_1)$ critical path of length at least five which contains the vertex $z_1$. It follows that $c_{1,\eta}(z_1,z'_1)=\eta$ since $z_1z'_1$ is just the fourth edge of this $(1,\eta,xy_1)$ critical path.
\item[(b)] If $c_{1,\eta}(p,p')=c_{j,\eta}(p,p')=1$, we have by Assumption \ref{asm:asm7} that $c_{1,\eta}(p,z_1)=c_{j,\eta}(p,z_1)=2$. Thus with respect to the partial coloring $c_{j,\eta}$, there exists a $(2,\eta,xy_j)$ critical path of length at least five which contains the vertex $z_1$. It follows that $c_{j,\eta}(z_1,z'_1)=\eta$ since $z_1z'_1$ is just the fourth edge of this $(2,\eta,xy_j)$ critical path.
\end{enumerate}
\end{obs}

\noindent \textbf{Local Recolorings:} If a partial coloring $h$ of $G$ is obtained from a partial coloring $c$ of $G$ by recoloring only certain edges incident on the vertices belonging to $N_{G}(p)-\{p',q\} = \{z_1,z_2,\ldots, z_k\}$ and also possibly the edge $pp'$, then $h$ is said to be obtained from $c$ by local recolorings.

The concept of local recolorings turns out to be crucial for the rest of the proof. The following lemma provides the main tool in this respect.

\begin{lem}
\label{lem:lem21}
\begin{enumerate}
\item[(a)] Let $c_{1,\eta}(p,p')=c_{j,\eta}(p,p')=2$. Also let $h_{1,\eta}$ be any valid coloring obtained from $c_{1,\eta}$ by recoloring only certain edges incident on the vertices belonging to $N_{G}(p)-\{p',q\} = \{z_1,z_2,\ldots, z_k\}$ and also possibly the edge $pp'$ (i.e., by only local recolorings). Then there exists a valid coloring $h_1$ of $G_1$ such that the valid coloring $h_{1,\eta}$ of $G_{1,\eta}$ is derivable from $h_1$.
\item[(b)] Let $c_{1,\eta}(p,p')=c_{j,\eta}(p,p')=1$. Also let $f_{j,\eta}$ be any valid coloring obtained from $c_{j,\eta}$ by recoloring only certain edges incident on the vertices belonging to $N_{G}(p)-\{p',q\} = \{z_1,z_2,\ldots, z_k\}$ and also possibly the edge $pp'$ (i.e., by only local recolorings). Then there exists a valid coloring $f_j$ of $G_j$ such that the valid coloring $f_{j,\eta}$ of $G_{j,\eta}$ is derivable from $f_j$.
\end{enumerate}
\end{lem}
\begin{proof}
\begin{enumerate}
\item[(a)]

Recall that $\eta \neq 1,2$. If $\eta \notin F_x(c_1)$, then $c_{1,\eta}=c_1$. In this case we take $h_1=c_1$. Otherwise if $\eta \in F'_x(c_1)$, let $xy_k$ be the edge in $G_1$ such that $c_1(x,y_k)=\eta$. Note that $k \neq 1$. It is enough to show that we can extend the valid coloring $h_{1,\eta}$ of $G_{1,\eta}$ to a valid coloring $h_1$ of $G_1$ by assigning an appropriate color to the edge $xy_k$ (Reader may note that neither $pp'$ nor any edge incident on the vertices in $\{z_1,z_2,\ldots, z_k\}$ can be the edge $xy_k$ since $x \notin \{p,p',z_1,z_2,\ldots, z_k,z'_1,z'_2,\ldots, z'_k\}$ due to $Lemma$ \ref{lem:lem15}). Now assign color $\eta$ to the edge $xy_k$ to get a coloring $d$. If the coloring $d$ is valid we are done and we have $h_1=d$. If it is not valid then there has to be a bichromatic cycle created in $G_1$ with respect to the coloring $d$. The cycle has to be a $(\eta,\theta)$ bichromatic cycle, where $d(y_k,y'_k)=\theta$. Moreover we can infer that $\theta \in F_{x}(d)$. If $\theta \neq d(p,p')=2$, then let $d'=d$. Otherwise we have $\theta = d(p,p')=2$. Now there exists a color $\omega \neq 2,\eta$ that is a candidate for the edge $y_ky'_k$. Recolor the edge $y_ky'_k$ using color $\omega$ to get a coloring $d'$ of $G_1$. Now if $d'$ is a valid coloring, then we are done and we have $h_1=d'$. If it is not valid, then $d'(y_k,y'_k)\neq 2$. Let $d'(y_k,y'_k)=\beta \neq 2$. Moreover with respect to the coloring $d'$ there should be a $(\eta,\beta)$ bichromatic cycle. Also let $\alpha\ (\neq \eta)\ \in C-F_x(c_1)=C-F_x(d')$. Now if,

\begin{enumerate}
\item[(1)] $\beta = 1$.

\begin{clm}
\label{clm:clm8}
None of the $(1,\gamma,xy_1)$ critical paths, where $\gamma\ (\neq \eta)\ \in C- F_x(c_1)$ are altered in the process of getting the coloring $h_{1,\eta}$ from $c_{1,\eta}$.
\end{clm}
\begin{proof}
Recall that only the edges incident on vertices $z_i$, where $z_i \in N_{G}(p)-\{p',q\}$ and edge $pp'$ are possibly recolored to get the coloring $h_{1,\eta}$ of $G_{1,\eta}$ from $c_{1,\eta}$. Note that by these recolorings the $(1,\gamma,xy_1)$ critical path cannot be extended since $x,y_1 \notin \{p,p',z_1,z_2,\ldots, z_k,z'_1,z'_2,\ldots, z'_k\}$ due to $Lemma$ \ref{lem:lem15}. Now if any $(1,\gamma,xy_1)$ critical paths are altered then they have to contain the above mentioned edges. Note that none of the vertices in $\{z_1,z_2,\ldots, z_k\}$ or vertex $p'$ can be the end vertices $x$ or $y_1$ and hence any critical path containing the vertex $z_i$ or $p'$ should also contain the vertex $p$ since $degree_{G_{1,\eta}}(z_i)=2$. We can infer that with respect to the coloring $c_{1,\eta}$, the $(1,\gamma,xy_1)$ critical path passes through the vertex $p$. It follows that this critical path has to contain the edge $pz_1$ since $c_{1,\eta}(p,z_1)=1$ (from part $(a)$ of $Observation$ \ref{obs:obs4}). Now since $z_1 \in W_0$ (i.e., $degree_{G_{1,\eta}}(z_1)=2$), this implies that $c_{1,\eta}(z_1,z'_1)=\gamma$, a contradiciton since from part $(a)$ of $Observation$ \ref{obs:obs4}, we know that $c_{1,\eta}(z_1,z'_1)=\eta$. Thus there cannot be any $(1,\gamma,xy_1)$ critical path containing the edges incident on vertices $z_i$, where $z_i \in N_{G}(p)-\{p',q\}$ and edge $pp'$. Thus none of the $(1,\gamma,xy_1)$ critical paths, where $\gamma \in C- F_{x}(c_1)$, $\gamma \neq \eta$ are altered.
\end{proof}

Since $d'$ is not valid there has to be a $(\eta,1)$ bichromatic cycle that passes through the vertex $x$. Now recolor the edge $xy_k$ with color $\alpha$ to get a coloring $d''$. Now if still there is a bichromatic cycle, then it should contain the edge $xy_k$ and hence the edge $y_ky'_k$. Therefore it is a $(\alpha,1)$ bichromtic cycle. This implies by $Fact$ \ref{fact:fact2} that there existed a $(1,\alpha,xy_k)$ critical path with respect to the coloring $d'$ and hence with respect to the coloring $h_{1,\eta}$. But in view of $Claim$ \ref{clm:clm8}, there exists a $(1,\alpha,xy_1)$ critical path with respect to the coloring $h_{1,\eta}$, a contradiciton in view of $Fact$ \ref{fact:fact1}. Thus the coloring $d''$ is valid.

\item[(2)] $\beta \neq 1$. This implies that $\beta\ (\neq \eta)\ \in F'_{x}(d')$. Let $y_t \in N'_{G}(x)$ be such that $d'(x,y_t)=\beta$. Thus $d'(y_t,y'_t)=\eta$. Now recolor the edge $xy_k$ with color $\alpha \in C-F_{x}(d')$ to get a coloring $d''$. Note that $\alpha \neq \eta$ since $\eta \notin C-F_{x}(d')$. Now if still there is a bichromatic cycle, then it should contain the edge $xy_k$ and hence the edge $y_ky'_k$. Therefore it is a $(\alpha,\beta)$ bichromtic cycle. Thus the bichromatic cycle should contain the edge $xy_t$. Since $degree_{G_1}(y_t)=2$, the bichromatic cycle should contain the edge $y_ty'_t$. But by our assumption, $d''(y_t,y'_t)=d'(y_t,y'_t)=\eta \neq \alpha$, a contradiction. Thus the coloring $d''$ is valid.

\end{enumerate}

Now let $h_1=d''$. Thus we get a valid coloring of $G_1$ from $h_{1,\eta}$.

\item[(b)] The proof of this is similar to that of part $(a)$ with $G_j$, $c_j$, $y_j$ taking the roles of $G_1$, $c_1$, $y_1$ respectively and the colors $1$ and $2$ exchanging their roles.

\end{enumerate}

\renewcommand{\qedsymbol}{\bqed}
\end{proof}

\begin{lem}
\label{lem:lem22}
\begin{enumerate}
\item[(a)] If $c_{1,\eta}(p,p') =c_{j,\eta}(p,p')=2$, then with respect to the coloring $c_{1,\eta}$, $2 \notin S_{z_1z'_1}$. (Recall that by Assumption \ref{asm:asm7}, $\{c_{1,\eta}(p,z_1)\}= \{c_{j,\eta}(p,z_1)\}= \{1,2\}-\{c_{1,\eta}(p,p')\}=\{1\}$.)
\item[(b)] If $c_{j,\eta}(p,p') =c_{1,\eta}(p,p')=1$, then with respect to the coloring $c_{j,\eta}$, $1 \notin S_{z_1z'_1}$. (Recall that by Assumption \ref{asm:asm7}, $\{c_{j,\eta}(p,z_1)\}= \{c_{1,\eta}(p,z_1)\}= \{1,2\}-\{c_{j,\eta}(p,p')\}=\{2\}$.)
\end{enumerate}
\end{lem}
\begin{proof}
\begin{enumerate}
\item[(a)] Suppose not. That is $2 \in S_{z_1z'_1}$. Note that by part $(a)$ of $Observation$ \ref{obs:obs4}, we have $c_{j,\eta}(p,z_1)=1$ and $c_{j,\eta}(z_1,z'_1)=\eta$. Therefore there exists some $\theta \notin \{1,2,\eta\}$ missing in $S_{z_1z'_1}$. Now recolor edge $z_1z'_1$ with color $\theta$ to get a coloring $c'_{1,\eta}$. If the coloring $c'_{1,\eta}$ is valid, then let $c''_{1,\eta} =c'_{1,\eta}$. Otherwise a bichromatic cycle gets formed by the recoloring. Since $c'_{1,\eta}(p,z_1)=1$, it has to be a $(1,\theta)$ bichromatic cycle and it passes through the vertex $p$. Thus there exists $z_i \in N_{G}(p)-\{q,p'\}$ such that $c'_{1,\eta}(p,z_i)=\theta$ and $c'_{1,\eta}(z_i,z'_i)=1$.

Now there exists a color $\mu \notin \{1,\theta,2,\eta\}$ missing at $p$. Recolor the edge $pz_1$ with color $\mu$ to get a coloring $c''_{1,\eta}$. This clearly breaks the $(1,\theta)$ bichromatic cycle that existed with respect to $c'_{1,\eta}$. But if a new bichromatic cycle gets formed with respect to $c''_{1,\eta}$, then it has to contain vertex $z_1$ and therefore the edge $z_1z'_1$, implying that it has to be a $(\mu,\theta)$ bichromatic cycle since $c''_{1,\eta}(z_1,z'_1)= \theta$. This cycle passes through the vertex $p$ and hence passes through the vertex $z_i$ since $c''_{1,\eta}(p,z_i)= \theta$, implying that $c''_{1,\eta}(z_i,z'_i)=\mu$, a contradiction since $c''_{1,\eta}(z_i,z'_i)=1$. Thus the coloring $c''_{1,\eta}$ is valid.

Note that we have possibly changed the colors of the edges $pz_1$ and $z_1z'_1$ to get $c''_{1,\eta}$ from $c_{1,\eta}$ (i.e., only local recolorings are done). Therefore by part $(a)$ of $Lemma$ \ref{lem:lem21} we infer that there exists a coloring $c''_1$ of $G_1$ such that $c''_{1,\eta}$ is derivable from $c''_1$. \emph{It follows from $Critical\ Path\ Property$ (i.e., $Lemma$ \ref{lem:lem5} or $Lemma$ \ref{lem:lem7}) that there exists a $(1,\eta,xy_1)$ critical path with respect to the coloring $c''_{1,\eta}$}. On the other hand recall that with respect to $c_{1,\eta}$ there existed a $(1,\eta,xy_1)$ critical path passing through $pz_1$ and $z_1z'_1$ (by part $(a)$ of $Observation$ \ref{obs:obs4}). But while getting $c''_{1,\eta}$ from $c_{1,\eta}$ we have indeed changed the color of at least one of the edges $pz_1$ or $z_1z'_1$ using a color other than $1$ and $\eta$. It follows that the $(1,\eta)$ maximal bichromatic path which contains the vertex $x$ ends at either vertex $p$ or $z_1$. Noting that $p,z_1 \neq y_1$, we infer by $Fact$ \ref{fact:fact1} that there cannot be a $(1,\eta,xy_1)$ critical path with respect to the coloring $c''_{1,\eta}$, a contradiciton.

\item[(b)] The proof of this is similar to that of part $(a)$ with $G_{j,\eta}$, $c_{j,\eta}$, $y_j$ taking the roles of $G_{1,\eta}$, $c_{1,\eta}$ and $y_1$ respectively and the colors $1$ and $2$ exchanging their roles.
\end{enumerate}

\renewcommand{\qedsymbol}{\bqed}
\end{proof}

\subsection{Getting a valid coloring that contradicts the Critical Path Property either from $c_1$ or from $c_j$}

In this section we will get the final contradiction in the following way: If $c_{1,\eta}(p,p')=c_{j,\eta}(p,p')=1$, then we will show that we can get a coloring $c'_j$ from $c_j$ that contradicts the Critical Path Property. Otherwise if $c_{1,\eta}(p,p')=c_{j,\eta}(p,p')=2$, then we will show that we can get a coloring $c'_1$ from $c_1$ that contradicts the $Critical\ Path\ Property$.

The two colorings $c_1$ and $c_j$ are very similar and hence we will only describe the way we get $c'_1$ from $c_1$. The same agruments can be imitated easily for $c_j$ by keeping the following correspondences in mind.
\begin{enumerate}
\item Vertex $y_1$ has same role as vertex $y_j$.
\item Colors $1$ and $2$ exchange their roles.
\item $(1,\gamma,xy_1)$ critical path has the same role as $(2,\gamma,xy_j)$ critical path, for $\gamma \in C'$. The $Critical\ Path\ Property$ of $c_1$ corresponds to that of $c_j$ (See Remarks after $Assumption$ \ref{asm:asm3} and $Assumption$ \ref{asm:asm6}).
\item Part $(a)$ of $Lemma$ \ref{lem:lem21} and $Lemma$ \ref{lem:lem22} applies to coloring $c_1$ while part $(b)$ applies to coloring $c_j$ in a corresponding way.
\item $Lemma$ \ref{lem:lem14} has the same role as $Lemma$ \ref{lem:lem13}.
\end{enumerate}

~~~~

We make the following assumption:

\begin{asm}
\label{asm:asm8}
Let $c_{1,\eta}(p,p')=c_{j,\eta}(p,p')=2$.
\end{asm}

\begin{obs}
\label{obs:obs5}
In view of $Assumption$ \ref{asm:asm8}, from $Observation$ \ref{obs:obs4} there exists a $(1,\eta,xy_1)$ critical path which contains the vertex $z_1$ with respect to the partial coloring $c_{1,\eta}$. Moreover this path is of length at least five. It follows that $c_{1,\eta}(p,z_1)=1$ and $c_{1,\eta}(z_1,z'_1)=\eta$. The first five vertices of the path are $x$, $q$, $p$, $z_1$, $z'_1$. Then clearly $z'_1 \neq y_1$ and hence is not a pendant vertex in $G_{1,\eta}$. Thus we have $S_{z_1z'_1} \neq \emptyset$ and $1 \in S_{z_1z'_1}$.
\end{obs}

\subsection*{Getting a valid coloring $d_1$ of $G_{1,\eta} - \{pz_1\}$ from $c_{1,\eta}$ by only local recolorings}

In view of $Lemma$ \ref{lem:lem22} and since $c_{1,\eta}(p,z_1)=1$, the color $2$ is a candidate for the edge $z_1z'_1$. We get a valid coloring $d_1$ of $G_{1,\eta} - \{pz_1\}$ from $c_{1,\eta}$ by removing the edge $pz_1$ and recoloring the edge $z_1z'_1$ by the color $2$. Note that $d_1$ is valid since $z_1$ is a pendant vertex in $G_{1,\eta} - \{pz_1\}$. Moreover we have broken the $(1,\eta,xy_1)$ critical path. Hence we have,
\begin{eqnarray}
\label{eqn:eqn6}
\mbox{\emph{With respect to the partial coloring $d_1$, there does not exists any $(1,\eta,xy_1)$ critical path.}}
\end{eqnarray}

\begin{lem}
\label{lem:lem23}
With respect to the partial coloring $d_1$ of $G_{1,\eta}$ , $\forall \gamma \in C-F_{p}(d_1)$, there exists a $(2,\gamma,pz_1)$ critical path. Since each of these critical paths has to contain the edge $pp'$, we can infer that $C-F_{p}(d_1) \subseteq S_{pp'}$.
\end{lem}
\begin{proof}
Suppose not. Then there exists a color $\gamma \in C-F_{p}(d_1)$ such that there is no $(2,\gamma,pz_1)$ critical path. By $Fact$ \ref{fact:fact2} color $\gamma$ is valid for the edge $pz_1$. Thus we get a valid coloring $d'_1$ of $G_{1,\eta}$ by coloring the edge $pz_1$ with color $\gamma$.

Note that we have possibly changed the colors of the edges $pz_1$ and $z_1z'_1$ to get $d'_{1}$ from $c_{1,\eta}$ (i.e., only local recolorings are done). Therefore by part $(a)$ of $Lemma$ \ref{lem:lem21} we infer that there exists a valid coloring of $G_1$ from which $d'_{1}$ can be derived. \emph{It follows from $Critical\ Path\ Property$ (i.e., $Lemma$ \ref{lem:lem5} or $Lemma$ \ref{lem:lem7}) that there exists a $(1,\eta,xy_1)$ critical path with respect to the coloring $d'_{1}$}. On the other hand recall that with respect to $c_{1,\eta}$ there existed a $(1,\eta,xy_1)$ critical path passing through $pz_1$ and $z_1z'_1$ (by $Observation$ \ref{obs:obs5}). But while getting $d'_{1}$ from $c_{1,\eta}$ we have indeed changed the color of the edges $z_1z'_1$ using the color $2 \notin \{1,\eta\}$. It follows that the $(1,\eta)$ maximal bichromatic path which contains the vertex $x$ ends at either vertex $p$ or $z_1$. Noting that $p,z_1 \notin y_1$, we infer that there cannot be a $(1,\eta,xy_1)$ critical path with respect to the coloring $d'_{1}$, a contradiciton.

\renewcommand{\qedsymbol}{\bqed}
\end{proof}

Note that with respect to $G_{1,\eta} - \{pz_1\}$, $\vert F_{p}(d_1) \vert \le \Delta -1$ and therefore $\vert C-F_{p}(d_1) \vert \ge 2$. But we know that color $1 \notin F_{p}(d_1)$. Since $\vert C-F_{p}(d_1) \vert \ge 2$, there exists a color $\mu \neq 1 \in C-F_{p}(d_1)$. Note that $\mu \neq 2$ also. The following observation is obvious in view of $Claim$ \ref{lem:lem23}:

\begin{obs}
\label{obs:obs6}
With respect to the partial coloring $d_1$ of $G_{1,\eta}$ , $1,\mu \notin F_{p}(d_1)$ and there exist $(2,1,pz_1)$ and $(2,\mu,pz_1)$ critical paths.
\end{obs}

\noindent \textbf{Selection of a special color $\theta$: } Since $\vert F_{p'}(d_1) \vert \le \Delta$, there exists a color $\theta$ missing at vertex $p'$. By $Lemma$ \ref{lem:lem23}, $\theta \notin C-F_{p}(d_1) \subseteq F_{p'}(d_1)$. Thus $\theta \in F_{p}(d_1)$. Clearly $\theta \neq 2$ since $2 \in F_{p'}$ and $\theta \neq 1,\mu$ because $1,\mu \notin F_{p}(d_1)$ and hence by $Lemma$ \ref{lem:lem23} we have $1,\mu \in S_{pp'}(d_1)$. Further $\theta \neq \eta$. This is because by $Lemma$ \ref{lem:lem14}, the $(\eta,2)$ maximal bichromatic path starts at vertex $q$ and contains the vertex $y'_j$. Clearly the first three vertices of this path are $q$, $p$, $p'$. Recall that the length of the segment of this path between vertices $q$ and $y'_j$ is at least three. Therefore $\eta \in S_{pp'}(d_1)$. Now without loss of generality let $d_1(p,z_2)=\theta (\neq 1,\eta,\mu,2)$. Note that $z_2$ is a vertex different from $z_1$.

Note that with respect to the coloring $c_{1,\eta}$, the $(1,\eta,xy_1)$ critical path passes through the vertex $z_1$ (by $(\ref{obs:obs5})$). This critical path cannot contain the vertex $z_2$. This is because if $z_2$ is an internal vertex of this critical path, then the edge $pz_2$ should be contained in the path, a contradiction since $c_{1,\eta}(p,z_2)=\theta \neq 1,\eta$. On the other hand if $z_2$ is an end vertex then it implies that $z_2 \in \{x,y_1\}$, a contradiction in view of $Lemma$ \ref{lem:lem15}. Thus vertex $z_2$ is not contained in the $(1,\eta,xy_1)$ critical path. While getting the coloring $d_1$ from $c_{1,\eta}$, this path was broken due to the recoloring of $z_1z'_1$ and $pz_1$. It follows that the $(1,\eta)$ maximal bichromatic path that starts at vertex $y_1$ does not contain vertex $z_2$. Thus we can infer that,

\begin{obs}
\label{obs:obs7}
With respect to the coloring $d_1$, there cannot exist a $(1,\eta,y_1,z_2)$ maximal bichromatic path.
\end{obs}

\subsection*{Getting a valid coloring $d_2$ of $G_{1,\eta} - \{pz_2\}$ from $d_1$ of $G_{1,\eta} - \{pz_1\}$ by only local recolorings}

\noindent We get a coloring $d''_1$ of $G_{1,\eta} - \{pz_1,pz_2\}$ from $d_1$ by discarding the edge $pz_2$. Note that the partial coloring $d''_1$ of $G_{1,\eta}$ is valid.

Now recolor the edge $pz_1$ with color the $special\ color$ $\theta$ to get a coloring $d_2$ of $G_{1,\eta} - \{pz_2\}$. Note that the color $\theta$ is a candidate for the edge $pz_1$ with respect to the coloring $d''_1$ since $d''_1(z_1,z'_1)=2$ and $\theta \notin F_p(d''_1)$ since we have removed the edge $pz_2$ (Recall that $d''_1(p,z_2)=\theta$). We claim that $d_2$ is valid also. Clearly if there is any bichromatic cycle created, then it has to be a $(\theta,2)$ bichromatic cycle since $d_2(z_1,z'_1) = 2$. Now this bichromatic cycle has to pass through vertex $p'$ since $d_2(p,p') = 2$. But by the definition of color $\theta$, it was not present at vertex $p'$. Thus there cannot be a $(\theta,2)$ bichromatic cycle. It follows that the partial coloring $d_2$ of $G_{1,\eta} - \{pz_2\}$ is valid. Recall that by $(\ref{eqn:eqn6})$ that there exists no $(1,\eta,xy_1)$ critical path with respect to $d_1$. Note that to get $d_2$ from $d_1$, we just assigned $\theta (\neq 1,2,\eta,\mu)$ to the edge $pz_1$ and removed the edge $pz_2$. Thus there is no chance of $(1,\eta,xy_1)$ critical path getting created with respect to $d_2$. Hence we have,
\begin{eqnarray}
\label{eqn:eqn7}
\mbox{\emph{With respect to the partial coloring $d_2$, there does not exists any $(1,\eta,xy_1)$ critical path.}}
\end{eqnarray}

\subsection*{Getting a valid coloring $c'_{1,\eta}$ of $G_{1,\eta}$ from $d_2$ of $G_{1,\eta} - \{pz_2\}$ by only local recolorings}

\noindent Now we will show that we can give a valid color for the edge $pz_2$ to get a valid coloring for the graph $G_{1,\eta}$. We claim the following:

\begin{lem}
\label{lem:lem24}
With respect to the coloring $d_2$ at least one of the colors $1$, $\mu$ is valid for the edge $pz_2$. (Recall that by $Observation$ \ref{obs:obs6}, $1,\mu \notin F_{p}(d_1)$ and therefore $1,\mu \notin F_{p}(d_2)$)
\end{lem}
\begin{proof}
Let $d_2(z_2,z'_2)=\sigma$. Now if,
\begin{enumerate}
\item $\sigma = 2$. Recolor the edge $pz_2$ using color 1 to get a coloring $d_3$. The coloring $d_3$ is valid because if a bichromatic cycle gets formed it has to be $(1,2)$ bichromatic cycle containing the vertex $p$ implying that there was a $(2,1,pz_2)$ critical path with respect to $d_2$. But by Observation \ref{obs:obs6}, there was a $(2,1,pz_1)$ critical path with respect to the coloring $d_1$ and hence with respect to the coloring $d_2$ (Note that to get $d_2$ from $d_1$, we just assigned $d_1(p,z_2)=\theta$ ($\neq 1,2,\eta,\mu$) to edge $pz_1$ and removed the edge $pz_2$. Thus the $(2,1,pz_1)$ critical path is not altered during this recoloring ). Thus in view of Fact \ref{fact:fact1}, there cannot be any $(2,1,pz_2)$ critical path with respect to $d_2$ since $z_1 \neq z_2$, a contradiction. Thus the coloring $d_3$ is valid.

\item $\sigma \in \{1,\mu\}$. Recolor the edge $pz_2$ using color $\{1,\mu\}-\{\sigma\}$ to get a coloring $d_3$. The coloring $d_3$ will be valid because if a bichromatic cycle gets formed it has to be $(1,\mu)$ bichromatic cycle containing the vertex $p$. But since color $\sigma \in \{1,\mu\}$ is not present at vertex $p$, such a bichromatic cycle is not possible.

\item $\sigma \notin \{1,2,\mu\}$. Recolor the edge $pz_2$ using color 1 to get a coloring $d'_2$. If the coloring $d'_2$ is valid, then let $d_3 = d'_2$. Otherwise if the coloring $d'_2$ is not valid, then there has to be a $(\sigma,1)$ bichromatic cycle. Now let $d'_2(p,z_j)=\sigma$. Then the bichromatic cycle passes through the vertex $z_j$ and hence $d'_2(z_j,z'_j)=1$, since $degree_{G}(z_j)=2$. Now we recolor edge $pz_2$ with color $\mu$ to get a coloring $d_3$. If there is a bichromatic cycle formed with respect to the coloring $d_3$, then it has to be a $(\mu,\sigma)$ bichromatic cycle and hence it passes through the vertex $z_j$. But color $\mu$ is not present at $z_j$ since $d'_2(z_j,z'_j)=1$. Thus there cannot be any $(\mu,\sigma)$ bichromatic cycle. Hence the coloring $d_3$ is valid.
\end{enumerate}

Thus either color $1$ or $\mu$ is valid for the edge $pz_2$.

\renewcommand{\qedsymbol}{\bqed}
\end{proof}

To get the coloring $d_3$ from $d_2$ we have only given a valid color for the edge $pz_2$ and have not altered the color of any other edge (i.e., only local recolorings are done). Recall that $d_2$ does not have any $(1,\eta,xy_1)$ critical path (by $(\ref{eqn:eqn7})$). Note that $d_3(x,q)=1$ and $d_3(q,p)=\eta$. If we give color $\mu \neq 1,\eta$ to the edge $pz_2$, there is no chance of a $(1,\eta,xy_1)$ critical path getting formed in $d_3$. On the other hand, by giving color 1 to the edge $pz_2$ if a $(1,\eta,xy_1)$ critical path gets formed, then it means that there exists a $(1,\eta,y_1,z_2)$ maximal bichromatic path with respect to $d_2$ and hence with respect to $d_1$. But by $Observation$ \ref{obs:obs7} such a bichromatic path does not exist. Now let $c'_{1,\eta}=d_3$. Thus we have,
\begin{eqnarray}
\label{eqn:eqn8}
\mbox{\emph{With respect to the valid coloring $c'_{1,\eta}$ of $G_{1,\eta}$, there does not exists any $(1,\eta,xy_1)$ critical path.}}
\end{eqnarray}

In getting $c'_{1,\eta}$ from $c_{1,\eta}$ we have done only local recolorings and thus by $Lemma$ \ref{lem:lem21} $c'_{1,\eta}$ can be derived from some valid coloring $c'_1$ of $G_1$. Note that we have not changed the color of the edge $y_1y'_1$ while getting $c'_{1,\eta}$ from $c_{1,\eta}$ since $y_1 \notin \{p,p',z_1,\ldots,z_k,z'_1,\ldots,z'_k\}$ (by $Lemma$ \ref{lem:lem15}). Thus $c'_{1,\eta}(y_1,y'_1)=1$. It follows that the $Critical\ Path\ Property$ of $c'_{1,\eta}$ is the same as $Critical\ Path\ Property$ of $c_{1,\eta}$. This implies that there exists a $(1,\eta,xy_1)$ critical path with respect to the coloring $c'_{1,\eta}$, a contradiction in view of $(\ref{eqn:eqn8})$.

This completes the proof.
\renewcommand{\qedsymbol}{\bqed}
\end{proof}


\end{document}